\documentclass[11pt,twoside,a4paper]{amsart}
\usepackage{amssymb}
\usepackage{graphicx}

\date{\today}

\usepackage[latin1]{inputenc}
\usepackage[T1]{fontenc}

\usepackage{geometry}
\def\T{T}
\def\k{{\kappa}}

\def\End{{\rm End}}
\def\dbar{\bar\partial}
\def\R{{\mathbf R}}
\def\C{{\mathbf C}}
\def\N{{\mathbf N}} 
\def\Z{{\mathbf Z}}
\def\Im{{\rm Im\, }}

\def\1{\mathbf 1}
\def\m{{\mathfrak m}}

\def\a{{\mathfrak a}}
\def\I{{\mathcal I}}
\def\lcm{\text{lcm}}
\def\Ok{\mathcal{O}}

\DeclareMathOperator{\sgn}{sgn}

\DeclareMathOperator{\codim}{codim}
\DeclareMathOperator{\Vol}{Vol}
\DeclareMathOperator{\res}{res}
\DeclareMathOperator{\Hom}{Hom}

\newtheorem{thm}{Theorem}[section]
\newtheorem{lma}[thm]{Lemma}
\newtheorem{cor}[thm]{Corollary}
\newtheorem{prop}[thm]{Proposition}

\newtheorem{claim}[thm]{Claim}

\theoremstyle{definition}

\theoremstyle{remark}

\newtheorem{preremark}[thm]{Remark}
\newtheorem{preex}[thm]{Example}

\newenvironment{remark}{\begin{preremark}}{\qed\end{preremark}}
\newenvironment{ex}{\begin{preex}}{\qed\end{preex}}

\numberwithin{equation}{section}

\numberwithin{figure}{section}

\begin{document}

\title
[On the fundamental class]
{On a representation of the fundamental class of an ideal due to
  Lejeune-Jalabert}

\date{\today}
\thanks{The author is supported by the Swedish 
  Research Council.}

\author
{Elizabeth Wulcan}

\address{Department of Mathematics\\Chalmers University of Technology and the University of Gothenburg\\S-412 96 Gothenburg\\SWEDEN}

\email{wulcan@chalmers.se}




\maketitle

\begin{abstract}



Lejeune-Jalabert showed that the fundamental class
of a Cohen-Macaulay ideal $\a\subset \Ok_0$ admits a
representation as a residue, constructed from a free resolution of
$\a$, of a certain differential form coming from the resolution.
We give an explicit description of this differential form in the case
where the free resolution is the Scarf resolution of a generic monomial
ideal. As a consequence we get a new proof of
Lejeune-Jalabert's result in this case.

\end{abstract}

\section{Introduction}

In \cite{L-J} Lejeune-Jalabert showed that the fundamental class
of a Cohen-Macaulay ideal $\a$ in the ring of germs of holomorphic
functions $\Ok_0$ at $0\in\C^n$ admits a
representation as a residue, constructed from a free resolution of $\a$,
of a certain differential form coming from the resolution, see
also \cite{L-J2,AL}. 
This representation generalizes the well-known fact that if $\a$ is 
generated by a regular sequence $f=(f_1,\ldots, f_n)$, then 
\begin{equation}\label{groth}
\res_f(df_n\wedge\cdots\wedge df_1)= 
\frac{1}{(2\pi i)^n}\int_\Gamma \frac{df_n\wedge\cdots\wedge df_1}{f_n\cdots f_1}
=\dim_\C(\Ok_0/\a), 
\end{equation}
where $\res_f$ is the \emph{Grothendieck residue} of $f_1,\ldots, f_n$
and $\Gamma$ is the real $n$-cycle defined by $\{|f_j|=\epsilon\}$ for some $\epsilon$ such that $f_j$
are defined in a neighborhood of $\{|f_j|\leq\epsilon\}$ and oriented
by $d(\arg f_1)\wedge \cdots\wedge d(\arg f_n)\geq 0$, see
\cite[Chapter~5.2]{GH}.

\smallskip

We will present a formulation of Lejeune-Jalabert's result in terms of 
currents. 
Recall that 
the \emph{fundamental cycle} of $\a$ 
is the cycle 
\begin{equation*}
[\a]=\sum m_j [Z_j], 
\end{equation*}
where $Z_j$ are the irreducible components of the variety $Z$ of
$\a$, and $m_j$ are the \emph{geometric multiplicities} of $Z_j$
in $Z$, defined as the length of the Artinian ring $\Ok_{Z_j, Z}$,
see, e.g., 
\cite[Chapter~1.5]{F}. In particular, if $Z=\{0\}$, then  
$[\a]=\dim_\C(\Ok_0/\a)[\{0\}]$.

Assume that 
\begin{equation}\label{upplosning}
0\to E_p\stackrel{\varphi_p}{\longrightarrow}\ldots\stackrel{\varphi_2}{\longrightarrow} E_1\stackrel{\varphi_1}{\longrightarrow}
E_0
\end{equation}
is a free resolution of 
$\Ok_0/\a$ of minimal length $p=\codim \a$; here the $E_k$ are free
$\Ok_0$-modules and $E_0\cong \Ok_0$. 
In \cite{AW} together with Andersson we constructed from \eqref{upplosning} a
(residue) current $R$, which has support on $Z$,
takes values in $E_p$, is of bidegree $(0,p)$, and can be thought of
as a current version of Lejeune-Jalabert's residue, cf.\ \cite[Section~6.3]{LW2}.  
Given bases of $E_k$, let $d\varphi_k$ be the $\Hom(E_k,E_{k-1})$-valued $(1,0)$-form with entries 
$(d\varphi_k)_{ij}=d(\varphi_k)_{ij}$ if $(\varphi_k)_{ij}$ are the
entries of $\varphi_k$ and let $d\varphi$ denote the
$E_p^*$-valued $(p,0)$-form
\[
d\varphi:=d\varphi_1\wedge \cdots\wedge d\varphi_p. 
\]
Now, identifying the fundamental cycle $[\a]$ with the current of integration along
$[\a]$, see, e.g., \cite[Chapter~III.2.B]{Dem}, Theorem~1.1 in
\cite{LW2} states that
$[\a]$ admits the factorization 
\begin{equation}\label{faktoriseringen} 
[\a]= \frac{1}{p!(-2\pi i)^p} d\varphi \wedge R. \footnote{For a
  discussion of the sign in \eqref{faktoriseringen}, see Section~2.6 in \cite{LW2}.}
\end{equation}
This should be thought of as a current version of
Lejeune-Jalabert's result in \cite{L-J}; in particular, the differential form
$d\varphi$ is the same form that appears in her paper. 
In fact, using residue theory the factorization
\eqref{faktoriseringen} can be obtained from 
Lejeune-Jalabert's result and vice versa; for a discussion of this, as
well as the formulation of Lejeune-Jalabert's result, we refer to 
Section ~6.3 in \cite{LW2}.
In \cite{LW2} we also give a direct proof of \eqref{faktoriseringen}
that does not rely on
\cite{L-J} and 
that extends to pure-dimensional ideal sheaves.

Assume that $\a$ is generated by a regular
sequence $f=(f_1,\ldots, f_p)$ and let $E_\bullet, \varphi_\bullet$ be the associated Koszul
complex, i.e., let $E$ be a free $\Ok_0$-module of rank $p$ with basis
$e_1, \ldots, e_p$, let $E_k=\Lambda^kE$ with bases $e_{\mathcal
  I}=e_{i_k}\wedge\cdots\wedge e_{i_1}$, and let $\varphi_k$
be the contraction with $\sum f_je_j^*$. Then
\[
R=\dbar\frac{1}{f_1}\wedge\cdots\wedge\dbar\frac{1}{f_p}
e_\emptyset^*\otimes e_{\{1,\ldots, p\}},
\]
where $R^f_{CH}=\dbar(1/f_1)\wedge\cdots\wedge\dbar(1/f_p)$ is the
classical \emph{Coleff-Herrera residue current} 
of $f$,
introduced in \cite{CH}, and $e_\emptyset$ denotes the basis element
of $\Lambda^0E\cong \Ok_0$. 
Moreover $d\varphi= p! df_p\wedge\cdots\wedge df_1 e_{\{1,\ldots, p\}}^*\otimes e_\emptyset$ 
and thus \eqref{faktoriseringen} reads
\begin{equation}\label{CHfaktor} 
[\a]= \frac{1}{(2\pi i)^p} \dbar\frac{1}{f_1}\wedge\cdots\wedge\dbar\frac{1}{f_p}
\wedge df_p\wedge \cdots\wedge df_1. 
\end{equation} 
This factorization of $[\a]$ can be seen as a current version of
\eqref{groth} 
and also as a generalization of the classical Poincar\'e-Lelong
formula 
\begin{equation}\label{plelong}
[f=0]=\frac{1}{2\pi i} \dbar\partial \log |f|^2=\frac{1}{2\pi i}
\dbar\frac{1}{f}\wedge df,
\end{equation}
where  $[f=0]$ is the current
of integration along the zero set of $f$, counted with
multiplicities. 
It appeared already 
in \cite{CH}, and in 
\cite{DP} Demailly and Passare proved an extension to locally complete
intersection ideal sheaves.

\smallskip 

In \cite{LW2} and \cite{L-J}, the factorization \eqref{faktoriseringen} is proved by
comparing $R$ and $d\varphi$ to a residue and differential form,
respectively, 
constructed from a certain Koszul
complex; in \cite{LW2} this is done using a recent comparison formula
for residue currents due to L\"ark\"ang, ~\cite{L}.

To explicitly describe the factors in \eqref{faktoriseringen}, however, seems to
be a delicate problem in general. In this note we compute the form 
$d\varphi$ when $E_\bullet, \varphi_\bullet$ is a certain resolution of a
monomial ideal. 
More precisely, 
let $A$ be the ring $\Ok_0$ of holomorphic germs at the origin in
$\C^n$ with coordinates $z_1,\ldots, z_n$, or let $A$ be the polynomial
ring $\C[z_1,\ldots, z_n]$. 
We then give an explicit description of the form 
\begin{equation}\label{huvud} 
d\varphi= 
\sum_\sigma 
\frac{\partial \varphi_1}{\partial z_{\sigma(1)}}dz_{\sigma(1)}\wedge
\cdots \wedge 
\frac{\partial \varphi_n}{\partial z_{\sigma(n)}}dz_{\sigma(n)}
\end{equation} 
when $E_\bullet, \varphi_\bullet$ is the \emph{Scarf resolution}, introduced in \cite{BPS}, of
an Artinian, i.e., zero-dimensional, \emph{generic} monomial ideal $M$
in $A$, see Section ~\ref{scarfcomplex} for definitions. Here the sum is
over all permutations $\sigma$ of
$\{1,\ldots, n\}$. 
It turns out that each summand in \eqref{huvud} is
a vector of monomials (times $dz_n\wedge\cdots \wedge dz_1$) whose coefficients have a neat description in
terms of the so-called staircase of $M$ and sum up to the geometric multiplicity
of $M$, see Theorem ~\ref{main} below. This can be seen as a far-reaching
generalization of the fact that the coefficient of $d(z^a)$ 
equals $a$, which is the geometric multiplicity of the principal ideal
$(z^a)$, cf.\ Example ~\ref{dimett} below.  Thus, in a sense, the
fundamental class of $M$ is captured already
by the form $d\varphi$. 
In the case of the Scarf resolution we 
recently, together with L\"ark\"ang, \cite{LW}, gave a complete 
description of the current $R$. Combining Theorem ~\ref{main} below with
Theorem ~1.1 in \cite{LW} we obtain a new proof of \eqref{faktoriseringen}
in this case, cf.\ Corollary ~\ref{foljd} below.

Let us describe our result in more detail. 
Let $M$ be an Artinian monomial ideal in $A$. By the \emph{staircase}
$S=S_M$ of $M$ we mean the set 
\begin{equation}\label{trappdef}
S=\overline{\{(x_1,\ldots, x_n)\in\R_{> 0}^n | z_1^{\lfloor x_1 \rfloor}\cdots  z_n^{\lfloor x_n \rfloor} \notin  M\}}\subset\R_{>0}^n. 
\end{equation} 
Here $\lfloor x \rfloor$ denotes the largest integer $\leq x$. 
The name is motivated by the shape of $S$, cf.\ Figures ~5.1, ~5.2,
and ~6.1. 
We will refer to the finitely many maximal elements in $S$, with respect to the
standard partial order on $\R^n$, as \emph{outer
  corners}.

The Scarf resolution $E_\bullet,\varphi_\bullet$ of $M$ is encoded in the \emph{Scarf complex},
$\Delta_M$, 
which is a labeled simplicial complex 
of dimension $n-1$
with one vertex for each minimal monomial generator of $M$ and one top-dimensional
simplex for each outer corner of $S$, see Section
~\ref{scarfcomplex}. 
The rank of $E_k$  equals the number of
$(k-1)$-dimensional simplices in $\Delta_M$. 
In particular, $E_\bullet,\varphi_\bullet$ ends at level $n$ and the rank of
$E_n$ equals the number of outer corners of $S$. 
Thus 
$d\varphi$ is
a vector with one entry for each outer corner of $S$.

For our description of $d\varphi$ we need to introduce certain
partitions of $S$. 
Given a permutation $\sigma$ of
$\{1,\ldots, n\}$ let $\geq_\sigma$ be the lexicographical order
induced by $\sigma$, i.e, $\alpha=(\alpha_1,\ldots, \alpha_n) \geq_\sigma 
\beta=(\beta_1,\ldots, \beta_n)$ if (and only if) 
for some
$1\leq k\leq  n$, $\alpha_{\sigma(\ell)}=\beta_{\sigma(\ell)}$ for
$1\leq \ell\leq k-1$ and
$\alpha_{\sigma(k)}>\beta_{\sigma(k)}$, or $\alpha=\beta$. If
$\alpha\neq \beta$ we write $\alpha>_\sigma\beta$. 
Let $\alpha^1\geq_\sigma \ldots\geq_\sigma \alpha^M$ be the total ordering of the outer corners
induced by $\geq_\sigma$, and define inductively 
\begin{eqnarray*}
S_{\sigma, \alpha^1}& =&\{x\in S \mid x\leq \alpha_1\}\\
& \vdots & \\
S_{\sigma, \alpha^k} & =& \{x\in S\setminus (S_{\sigma,
    \alpha^1}\cup\cdots\cup S_{\sigma, \alpha^{k-1}}) \mid x\leq
\alpha_k\} \\
& \vdots & 
\end{eqnarray*}
For a fixed $\sigma$,
$\{S_{\sigma,\alpha}\}_\alpha$ provides a partition of $S$,
cf.\ Section ~\ref{trappsektion}.

\begin{thm}\label{main}
Let $M$ be an Artinian generic monomial ideal in $A$, and let
\begin{equation}\label{andetag}
0\to E_n\stackrel{\varphi_n}{\longrightarrow}\ldots\stackrel{\varphi_2}{\longrightarrow} E_1\stackrel{\varphi_1}{\longrightarrow}
E_0
\end{equation}
 be the Scarf resolution of $A/M$. 
Then 
\begin{equation}\label{spor}
d_\sigma\varphi :=
\frac{\partial \varphi_1}{\partial z_{\sigma(1)}}dz_{\sigma(1)}\wedge 
\cdots \wedge  
\frac{\partial \varphi_n}{\partial z_{\sigma(n)}}dz_{\sigma(n)}
\end{equation}
has one entry $(d_\sigma\varphi)_\alpha$ for each outer corner $\alpha$ of the
staircase $S$ of $M$ and 
\begin{equation}\label{snor}
(d_\sigma \varphi)_\alpha=\sgn(\alpha)\Vol (S_{\sigma,\alpha})
z^{\alpha-\1} dz, 
\end{equation} 
where 
$\sgn(\alpha)=\pm 1$ comes from the orientation of the Scarf
complex, 
$dz=dz_n\wedge\cdots\wedge
dz_1$, and $z^{\alpha-\1}= 
z_1^{\alpha_1-1} \cdots z_n^{\alpha_n-1}$ if $\alpha=(\alpha_1,\ldots,
\alpha_n)$. 
\end{thm}

The sign $\sgn(\alpha)$ will be specified in Section ~\ref{strut}
below.

Theorem ~5.1 in \cite{LW} asserts that the residue current associated
with the Scarf resolution 
has one entry 
\begin{equation}\label{poor}
R_\alpha=\sgn(\alpha) ~
\dbar \frac{1}{z_1^{\alpha_1}} \wedge\cdots\wedge 
\dbar \frac{1}{z_n^{\alpha_n}} 
\end{equation} 
for each outer corner $\alpha$ of $S$. 
Since $(1/2\pi i)\dbar(1/z^a)\wedge d(z^a)=[z=0]$, cf.\ \eqref{plelong}, we conclude
from \eqref{snor} and \eqref{poor} that 
\begin{equation*}
\frac{1}{(-2\pi i)^n}d_\sigma \varphi 
\wedge R =
\sum_\alpha \Vol (S_{\sigma, \alpha}) [0] = 
\Vol (S) [0]. 
\end{equation*}
Note that $\Vol (S)$ equals the number of monomials that are not in
$M$. Since these monomials form a basis for $A/M$,
$\Vol(S)$ equals the geometric multiplicity $\dim_\C(A/M)$ of $M$. Thus we
get the following version of \eqref{faktoriseringen}.
\begin{cor}\label{foljd} 
Let $M$ and \eqref{andetag} be as in Theorem ~\ref{main} and let $R$
be the associated residue current. Then 
\begin{equation}\label{fluff}
\frac{1}{(-2\pi i)^n}\frac{\partial \varphi_1}{\partial z_{\sigma(1)}}dz_{\sigma(1)}\wedge \cdots \wedge 
\frac{\partial \varphi_n}{\partial z_{\sigma(n)}}dz_{\sigma(n)}\wedge R=
[M]. 
\end{equation} 
\end{cor} 
Summing over all permutations $\sigma$ we get back \eqref{faktoriseringen}. 
In fact, as was recently pointed out to us by Jan Stevens, 
if \eqref{andetag} 
is any free resolution of an Artinian ideal and $R$ is the
associated residue current, then  
the left hand side of \eqref{fluff} is independent of $\sigma$, see
Proposition ~\ref{artinartin}.

\medskip 

The core of the proof of Theorem ~\ref{main} is an alternative
description of the $S_{\sigma, \alpha}$ as certain cuboids, see Lemma
~\ref{ratblock}. Given this description it is fairly straightforward to see that
the volumes of the $S_{\sigma, \alpha}$ are precisely the coefficients of
the monomials in $d_\sigma\varphi$; this is done in Section
\ref{raknautdf}.

\smallskip

We suspect that Theorem
~\ref{main} extends to a more general setting
than the one above. If $M$ is an Artinian non-generic monomial ideal, we
can still construct the partitions $\{S_{\sigma,\alpha}\}_\alpha$. 
The elements $S_{\sigma, \alpha}$ will, however, no longer be cuboids in
general. Also, the computation of $d\varphi$ is more delicate in general. 
In Example
~\ref{amsterdam} we compute $d\varphi$ for a non-generic monomial ideal
for which the \emph{hull resolution}, introduced in \cite{BaS}, is minimal, and show
that Theorem ~\ref{main} holds in this case. On the other hand, in Example
~\ref{motex} 
we consider a monomial ideal for which the hull resolution is not a
minimal resolution and where Theorem ~\ref{main} fails to
hold.

\smallskip 

The paper is organized as follows. In Section ~\ref{trappsektion} and
~\ref{scarfcomplex} we provide some background on staircases of
monomial ideals and the Scarf complex, respectively. 
The proof of Theorem ~\ref{main} occupies Section
~\ref{beviset} and in Section ~\ref{exempel} we illustrate the proof by
some examples. Finally, in Section ~\ref{general} we consider 
resolutions of non-generic Artinian monomial ideals and look at some
examples. We also show that the left hand side of \eqref{fluff} is
independent of $\sigma$ in general.

\subsection*{Acknowledgment}
I want to thank Mats Andersson, Richard L\"ark\"ang, and Jan Stevens 
for valuable discussions. Also, thanks to the referee for many 
valuable comments and suggestions that have 
 helped to
improve the exposition of the paper.

\section{Staircases}\label{trappsektion} 

We let $\geq$ denote the standard partial order on $\R^n$, i.e.,
$a=(a_1,\ldots,a_n)\geq b=(b_1,\ldots, b_n)$ if and only if
$a_\ell\geq b_\ell$ for $\ell=1,\ldots, n$. If $a\geq b$ and $a\neq b$ we
write $a>b$. If $a_\ell>b_\ell$ for all $\ell$
we write $a\succ b$. 
Throughout we let $A=A_n$ denote the ring $\C[z_1,\ldots, z_n]$ or the
ring $\Ok_0$ of holomorphic germs at $0\in\C^n_{z_1,\ldots, z_n}$. 
For $a=(a_1,\ldots,a_n)\in \N^n$, where $\N=0,1,\ldots$, we use the shorthand notation
$z^a$ for the monomial $z_1^{a_1}\cdots z_n^{a_n}$ in $A$. For a general
reference on (resolutions of) monomial ideals, see, e.g., \cite{MS}. 

Unless otherwise stated $M$ will be a monomial ideal in $A$, i.e., an
ideal generated by
monomials, and $S$ will be the staircase of $M$ as defined in 
\eqref{trappdef}. 
Note that $M$ is Artinian if and only if there are generators of the
form 
$z_i^{a_i}$, $a_i>0$, for $i=1,\ldots, n$, which is equivalent to that $S\subset
\{x\in \R^n_{>0}\mid x_i\leq a_i, i=1,\ldots, n\}$ for some $a_i$,
which in turn is equivalent to that $S$ is
bounded. 
Recall that the closure in \eqref{trappdef} is taken
in $\R^n_{>0}$; we will however often consider $S$ as a subset of
$\R^n$. 
As in the introduction we will refer to the maximal elements of $S$ as
\emph{outer corners}. The minimal elements of
$\overline{\R^n_{>0}\setminus S}\subset\R^n$ 
we will call \emph{inner corners}. 
Unless otherwise mentioned the closure $\overline A$ of a set $A$ is
taken in $\R^n$.

One can check that for any monomial ideal $M$ there is a unique
minimal set of exponents $B\subset\N^n$ such that the monomials $\{z^a\}_{a\in
  B}$ generate $M$. We refer to these monomials as \emph{minimal monomial
  generators} of $M$. 
Moreover
\begin{equation}\label{complement}
S=\R^n_{>0}\setminus \bigcup_{a\in B} (a +\R_{> 0}^n). 
\end{equation}
In particular, the inner corners of $S$ are precisely the elements in $B$.

Dually, $M$ can be described as an intersection of so-called
\emph{irreducible} monomial ideals, i.e., ideals 
generated by powers of variables; such an ideal can be described as
$\m^\alpha:=(z_i^{\alpha_i}|\alpha_i\geq 1)$, where
$\alpha=(\alpha_1,\ldots, \alpha_n)\in\N^n$ (more generally
an ideal is irreducible if it cannot be written as a non-trivial
intersection of two ideals).  
For every monomial ideal $M$ there is a unique minimal set $C\subset \N^n$ 
such that 
\begin{equation}\label{irreducible}
M=\bigcap_{\alpha\in C}\m^\alpha,  
\end{equation}
see, e.g., \cite[Theorem~5.27]{MS}. 
The ideal $M$ is Artinian if and only if each $\alpha\in C$
satisfies $\alpha\succ 0$ (i.e., $\alpha_\ell >0$ for each $\ell$). 
If $\alpha\succ 0$, then note that a monomial $z^b\not\in\m^\alpha$ if and
only if $b\prec \alpha$. 
It follows that 
\begin{equation}\label{alliance}
S=\bigcup_{\alpha\in C} \{x\in \R^n_{> 0}\mid x\leq \alpha\}. 
\end{equation}
In particular, the outer corners of $S$ are precisely the elements in $C$. 
If $M$ is not Artinian, then $S$ is not bounded in $\R^n$ and 
the representation \eqref{alliance} fails to hold. 
Note that \eqref{alliance} guarantees that for a fixed $\sigma$,
$\{S_{\sigma,\alpha}\}_\alpha$, as defined in the introduction, is a
partition of $S$.

\medskip

Inspired by \eqref{complement} we will call any set of this form a
staircase: 
Let $H$ be an affine subspace of $\R^n$ of the form 
\[
H=\{x_{\ell_1}=a_1,\ldots, x_{\ell_k}=a_k\}
\]
where $a_j\in\Z$. 
For $a\in \Z^n\cap H$ let $U_a=\{x\in H \mid x\succ a\}$. Note that $U_0^{\Z^n}$
is just the first (open) orthant $\R^n_{>0}$ in $\R^n$. 
We say that a set $S\subset H$ is a
\emph{staircase} if it is of the form 
\[
S=U_{a^0}\setminus \bigcup_{j=1}^s U_{a^j}
\]
for some $a^0,a^1, \ldots, a^s\in \Z^n\cap H$. We say that $a_0$ is the
\emph{origin} of $S$ and if $a^1,\ldots, a^s$ are chosen so $a^1,
\dots, a^s \geq a^0$ and $a_j\not\leq a_k$ for $j\neq k$, $1\leq
j,k\leq s$ we call them the \emph{inner
  corners} of $S$. We call the maximal elements of $S$ the
\emph{outer corners} of $S$. 
Since $\Z^n\cap H$ is a lattice, 
the outer corners are in ~$\Z^n\cap H$. 
Note that $S$ is a closed subset of $U_{a^0}$.

If $\pi$ is the projection $\R^n\to \R^{n-k}$ that maps $(x_1,\ldots,
x_n)$ to $(x_{j_1},\ldots, x_{j_{n-k}})$ if $\{j_1,\ldots,
j_{n-k}\}=\{1,\ldots, n\}\setminus \{\ell_1,\ldots, \ell_k\}$ and $\rho:\R^n\to \R^{n-k}$ is the affine map
$\rho(x)=\pi(x-a^0)$, let 
$M(S)=M_\rho(S)$ be the monomial ideal in $A_k$ that is generated
by $z^{\rho(a^j)}$, where $a^j$ are
the inner corners of $S$. Then the staircase of $M(S)$ equals $\rho(S)$.

For $\alpha\in L$, let $V_\alpha =V_\alpha^L =\{x\in H \mid x \leq \alpha\}$. If $M(S)$
is Artinian, then $S$ admits a representation analogous to
\eqref{alliance}, 
\begin{equation}\label{flinga}
S=U_{a^0} \cap \bigcup_\alpha V_\alpha,
\end{equation} 
where the union is taken over all outer corners of $S$.

Note that any set $S$ of the form \eqref{flinga} with $a^0, \alpha\in L$
is a staircase; indeed since $L$ is a lattice the minimal elements of
$\overline{U_{a^0}\setminus S}$
are in $L$.

\section{The Scarf complex}\label{scarfcomplex} 

For $a,b\in \R^n$, we will denote by $a\vee b$ the \emph{join} of $a$
and $b$, i.e., the unique $c$ such that $c\geq a, b$, and $c\leq d$ for
all $d\geq a, b$.

Let $M$ be an Artinian monomial ideal in $A$, with minimal monomial
generators $m_1=z^{a^1},\ldots, m_r=z^{a^r}$. 
The  \emph{Scarf complex} ~$\Delta=\Delta_{M}$ of ~$M$ was introduced by
Bayer-Peeva-Sturmfels, \cite{BPS}, based on previous work by
H. Scarf. 
It is the collection of
subsets $\I=\{i_1,\ldots, i_k\}\subset\{1,\ldots,r\}$ whose corresponding least common
multiple ~$m_{\I}:=\lcm(m_{i_1},\ldots, m_{i_k})=z^{a^{i_1}\vee\cdots
  \vee a^{i_k}}$ is unique, that is,
\[
\Delta=\{\I\subset\{1,\ldots,r\}|m_{\I}=m_{\I'}\Rightarrow \I=\I'\}.
\]
Clearly the vertices of $\Delta$ are the
minimal monomial generators of $M$, i.e., the inner corners of $S_M$. 
One can prove that the Scarf complex is a simplicial complex of
dimension at most ~$n-1$. We let $\Delta(k)$ denote the set of 
simplices in $\Delta$ with $k$ vertices, i.e., of dimension $k-1$. 
Moreover we label the faces $\I\subset\Delta$ by the monomials
$m_{\I}$. 
We will sometimes be sloppy and identify the faces in $\Delta$ with
their labels or exponents of the labels and write $m_{\I}$ or
$\alpha$ for the face with label
$m_{\I}=z^\alpha$ and 
$\{z^{a^{i_1}},\ldots,
z^{a^{i_k}}\}$ or $\{a^{i_1},\ldots, a^{i_k}\}$ for $\I=\{i_1,\ldots, i_k\}$.

The ideal ~$M$ is said to be \emph{generic} in the sense of \cite{BPS,MSY} if whenever two distinct
minimal generators ~$m_i$ and ~$m_j$ have the same positive degree in
some variable, then there is a third generator ~$m_k$ that
\emph{strictly divides} $m_{\{i,j\}}$, which means that ~$m_k$ divides
~$m_{\{i,j\}}/z_\ell$ for all variables ~$z_\ell$ dividing ~$m_{\{i,j\}}$. In particular, ~$M$ is generic if no two generators have the same positive degree in any variable.

If $M$ is generic, then  $\Delta$ has precisely
dimension $n-1$; it is a regular 
triangulation of the $(n-1)$-dimensional simplex, see
\cite[Corollary~5.5]{BPS}.  
The (labels of the) top-dimensional faces of $\Delta$ are
precisely the exponents $\alpha$ in a minimal irreducible decomposition
\eqref{irreducible} of $M$, i.e., the outer corners of $S_M$, see
\cite[Theorem~3.7]{BPS}.  

For $k=0,\ldots, n$, let $E_k$ be the free
$A$-module with basis $\{e_{\I}\}_{\I\in \Delta(k)}$ and let the
  differential $\varphi_k:E_k\to E_{k-1}$ be defined by 
\begin{equation}\label{blabla}
\varphi_k: 
e_{\I}\mapsto \sum_{j=1}^k
(-1)^{j-1}~\frac{m_{\I}}{m_{{\I}_j}}e_{{\I}_j}, 
\end{equation} 
where ${\I}_j$ denotes $\{i_1,\ldots, i_{j-1},i_{j+1}, \ldots, i_k\}$ if
${\I}=\{i_1, \ldots, i_k\}$. 
Then the complex $E_\bullet, \varphi_\bullet$ is exact and thus gives
a free resolution of the cokernel of $\varphi_0$, which with the
identification $E_0=A$ equals $A/M$, see \cite[Theorem~3.2]{BPS}. In fact, this
so-called \emph{Scarf resolution} is a minimal resolution of $A/M$,
i.e., for each ~$k$, ~$\varphi_k$ maps a basis
of ~$E_k$ to a minimal set of generators of $\Im \varphi_k$, see,
e.g., \cite[Corollary~1.5]{E2}. 
Originally, in \cite{BPS}, the situation $A=\C[z_1,\ldots, z_n]$ was
considered. 
However, since $\Ok_0$ is flat over $\C[z_1,\ldots, z_n]$,
see, e.g.,  \cite[Theorem~13.3.5]{Ta}, the complex
$E_\bullet,\varphi_\bullet$ is exact for $A=\Ok_0$ if and only if it is
exact for $A=\C[z_1,\ldots, z_n]$.

\subsection{The sign $\sgn(\alpha)$}\label{strut}
Let ${\I}=\{i_1,\ldots, i_n\}$ be a top-dimensional simplex in
$\Delta$ with label $\alpha$. Then there
is a unique permutation $\eta=\eta(\alpha)$ of $\{1,\ldots, n\}$ such that for
$1\leq \ell\leq n$ there 
is a unique vertex $i_{\eta(\ell)}$ of ${\I}$ such that 
$\alpha_\ell=a^{i_{\eta(\ell)}}_\ell$; we will refer to this vertex as the
$x_\ell$\emph{-vertex} of ${\I}$. To see this, first of all, since
$\alpha=\lcm(a^i)$, $a^i_\ell\leq \alpha_\ell$ and therefore there must be at
least one vertex $i$ of $\I$ such that $a^i_\ell=\alpha_\ell$. Assume that $i$ and $j$ are
vertices of ${\I}$ such that  $a^i_\ell=a^j_\ell=\alpha_\ell$. Then, since $M$ is
generic, there is a generator $z^b$ of $M$ that strictly
divides $z^{a^i\vee a^j}$. But then $a^i\vee a^j= a^i \vee a^j \vee b$ and
so $\{i,j\}$ is not in $\Delta$, which contradicts that $i$ and $j$
are both vertices of $\I$.

 We let $\sgn(\alpha)$ denote the sign of the permutation $\eta$. This is the sign
  that appears in \eqref{snor} in Theorem ~\ref{main} as well as in
  \eqref{poor}. 
We should remark
that we use a different sign convention in this paper than in
\cite{LW, LW2}, which corresponds 
to a different orientation of the $(n-1)$-simplex or,  
equivalently, to a different choice of
bases for the modules $E_k$, cf.\ \cite{AW, LW2}.

\subsection{The subcomplex $\Delta_{\sigma, a^1,\ldots, a^k}$}\label{delkomplex}

Given a permutation $\sigma$ of $\{1,\ldots, n\}$, and 
vertices $a^1,\ldots, a^k$ of $\Delta$, let $\Delta_{\sigma, a^1,
  \ldots, a^k}$ be the (possibly empty) subcomplex of $\Delta$, with
top-dimensional simplices $\alpha$ that satisfy
$\alpha_{\sigma(\ell)}=a^\ell_{\sigma(\ell)}$ for $\ell=1,\ldots, k$. 
In other words, the top-dimensional
simplices in $\Delta_{\sigma, a^1, \ldots, a^k}$ are the ones that have
$a^\ell$ as $x_{\sigma(\ell)}$-vertex for $\ell=1,\ldots, k$. 

Note that for each choice of permutation $\sigma$, $k\in\{1,\ldots,n\}$, and $\alpha\in\Delta(n)$ there is a unique
sequence $a^1,\ldots, a^k$ such that $\alpha\in \Delta_{\sigma, a^1, \ldots,
  a^k}$. 
Moreover, $\Delta_{\sigma, a^1, \ldots, a^{n}}$ is the unique simplex
in $\Delta(n)$ that satisfies $a^\ell_{\sigma(\ell)} =\alpha_{\sigma(\ell)}$ for
$\ell=1,\ldots, n$ if $\alpha$ is the label of $\Delta_{\sigma, a^1, \ldots, a^{n}}$.

We will write $\Delta_{\sigma, a^1,\ldots, a^k}^*$ for the subcomplex
of $\Delta_{\sigma, a^1,\ldots, a^k}$ consisting
of all faces in $\Delta_{\sigma, a^1,\ldots, a^k}$ that do not contain
$a^1,\ldots, a^{k-1}$, or $a^k$. 
Note that since $\Delta_{\sigma, a^1,\ldots, a^k}$ is simplicial, $\{b^1,\ldots, b^\ell\}$ is a
face of $\Delta_{\sigma, a^1,\ldots, a^k}^*$ if and only $\{a^1,\ldots, a^k, b^1,\ldots, b^\ell\}$ is a face of
$\Delta_{\sigma, a^1,\ldots, a^k}$.

\section{Proof of Theorem ~\ref{main}}\label{beviset} 
To prove the theorem we will first give an alternative description of
the $S_{\sigma,\alpha}$ as certain cuboids. Throughout this section we will assume that $E_\bullet,\varphi_\bullet$
is the Scarf resolution of a generic Artinian monomial ideal $M$ and 
we will use the notation from above.

\begin{lma}\label{ratblock}  
Assume that $\alpha$ is the label of the face ${\I}=\{i_1,\ldots,
i_n\}\in \Delta(n)$. Let $\eta$ be the permutation of $\{1,\ldots,
n\}$ associated with $\I$ as in Section ~\ref{strut}, and
set $\tau=\eta\circ \sigma$. 

Then $S_{\sigma,\alpha}$ is a cuboid with side lengths 
\[
a^{i_{\tau(1)}}_{\sigma(1)}, \big (a^{i_{\tau(1)}}\vee a^{i_{\tau(2)}}-a^{i_{\tau(1)}} \big )_{\sigma(2)}, \ldots,  \big (a^{i_{\tau(1)}}\vee\cdots\vee
a^{i_{\tau(n)}}-a^{i_{\tau(1)}}\vee\cdots\vee a^{i_{\tau({n-1})}} \big )_{\sigma(n)}.
\]
\end{lma}

In particular, 
\[
\Vol (S_{\sigma,\alpha})=
a^{i_{\tau(1)}}_{\sigma(1)}\times  \big (a^{i_{\tau(1)}}\vee a^{i_{\tau(2)}}-a^{i_{\tau(1)}} \big )_{\sigma(2)}\times \cdots \times  \big (a^{i_{\tau(1)}}\vee\cdots\vee
a^{i_{\tau(n)}}-a^{i_{\tau(1)}}\vee\cdots\vee a^{i_{\tau({n-1})}} \big )_{\sigma(n)}.
\]

\subsection{Proof of Lemma ~\ref{ratblock}}\label{ratblocksbevis}

To prove the lemma, let us first assume that $\sigma$ 
is the identity permutation and write
$S_\alpha=S_{\sigma,\alpha}$, $\Delta_{a^1,\ldots,
  a^k}=\Delta_{\sigma, a^1,\ldots, a^k}$ (so that $\Delta_{a^1,\ldots, a^k}$ is the subcomplex of $\Delta$ whose
top-dimensional simplices $\alpha$ satisfy $\alpha_\ell=a^\ell_\ell$
for $\ell=1,\ldots, k$), and  $\Delta_{a^1,\ldots,
  a^k}^*=\Delta_{\sigma, a^1,\ldots, a^k}^*$.

We will decompose $S$ in a seemingly different way. 
First we will construct certain lower-dimensional staircases with
corners in 
$\Delta$. 
Given an inner corner $a=(a_1,\ldots, a_n)$ of $S$ 
let 
\begin{equation}\label{krut}
\T_a=\{x\in S \mid x_1=a_1, x_\ell > a_\ell, \ell=2,\ldots, n\}. 
\end{equation}
Note that $\T_a$ is contained in the boundary $\partial S$ of $S$. 
Let $H_a$ be the hyperplane $\{x_1=a_1\}$, and let $U_a$ be defined as
in Section \ref{trappsektion}. Then note that $\T_a= U_a\cap S \subset 
H_a \cap S$.

\begin{claim}\label{lila}
Let $a$ be an inner corner of $S$. 
Then $\T_a=\emptyset$ if and only if $a_1=0$. If $T_a$ is non-empty, then it is a staircase in the hyperplane $H_a$ with origin
$a$. The outer corners of $\T_a$ are the top-dimensional faces of
$\Delta_a$. The inner corners are the lattice points  $a\vee b$,
where $b$ is a vertex in $\Delta_a^*$. 
\end{claim}

\begin{proof}
If $a_1=0$, then $H_a\cap S=\emptyset$, and thus $T_a$ is empty. 
In general, note that $T_a$ is empty exactly if $a+(0,1,\ldots, 1)\notin S$. If
$a_1>0$, so that  $a+(0,1,\ldots, 1)\in \R^n_{>0}$, 
this means that there is an inner
corner $c$ of $S$ such that $c\prec a+(0,1,\ldots, 1)$. In particular, 
$c\leq a$, which contradicts that $a$ is an inner corner. Thus
$T_a\neq \emptyset$ if $a_1>0$. 

\smallskip

Assume that $T_a$ is non-empty and that $\beta=(\beta_1,\ldots, \beta_n)$ is maximal in
$\T_a$. Since $\T_a\subset H_a$, $\beta_1=a_1$. Moreover, since $S$ is
of the form \eqref{alliance}, 
there is a maximal $\gamma\in S$ such that
$\gamma\geq\beta$. 
By the definition of
$\T_a$, $\gamma_\ell\geq \beta_\ell > a_\ell$ for $\ell=2,\ldots, n$ and
thus if $\gamma_1>a_1$, then $\gamma\succ a$, which contradicts that $\gamma\in S$. Hence  
$\gamma_1=a_1$ and, since $\gamma_\ell > a_\ell$ for
$\ell=2,\ldots, n$, $\gamma\in \T_a$. Since $\beta$ is maximal in $\T_a$,
$\gamma=\beta$, which  means that, in fact, $\beta$ is maximal in $S$
and thus $\beta\in\Delta(n)$. Since 
$\beta_1=a_1$, $\beta\in\Delta_a(n)$ by the definition of $\Delta_a$.

On the other hand, if $\beta\in \Delta_a(n)$, then $\beta$ is maximal
in $S$ and contained in $\T_a\subset S$, and thus it is maximal in $\T_a$. We
conclude that the maximal elements in $\T_a$ are the
top-dimensional faces of $\Delta_a$.

Since $S$ is of the form \eqref{alliance}, 
\begin{equation*}
\T_a=U_a 
\cap \bigcup_{\beta\in \Delta_a(n)} \{x\mid x\leq \beta\}, 
\end{equation*} 
which in light of \eqref{flinga} means that $\T_a$ is a staircase in
$H_a$ with origin $a$ and outer corners $\Delta_a(n)$.

\smallskip

It remains to describe the inner corners of $\T_a$. 
Assume that $\beta=(a_1,\beta_2,\ldots, \beta_n)$ is an inner corner
of $\T_a$,
i.e., $\beta$ is minimal in $\overline{U_a\setminus S}$. 
This means that any $\tilde\beta$, such that
$\tilde\beta_\ell>\beta_\ell$ for $\ell=2,\ldots, n$, is not contained
in $S$, which implies that there is an inner corner $b$ of $S$ such
that $b\prec \tilde \beta$ for any such $\tilde\beta$. In particular, $b\leq \beta$ and
$b_1<a_1$. To conclude, there is an inner corner $b\neq a$ of $S$ such
that $b\leq \beta$. 
Now $a\vee b\in \overline{U_a\setminus S}$, since $a\vee b \geq a$,
$(a\vee b)_1=a_1$, and $z^{a\vee b}\in M$. 
Moreover, $a\vee b\leq a\vee \beta \leq \beta$, since $b\leq \beta$, and, since
$\beta$ by assumption is minimal in $\overline{U_a\setminus S}$, it
follows that 
$\beta=a\vee b$. 
Now, since $\T_a$ is a staircase there is a maximal
$\alpha\in\Delta_a(n)$ such that $b\leq \beta \leq \alpha$. It follows
that $b$ is a vertex of $\alpha$, i.e., $b\in \Delta_a^*(1)$.

Conversely, pick $b\in\Delta_a^*(1)$ and let $\beta=a\vee b$. Since
$a$ and $b$ are inner corners of $S$, $z^\beta\in M$ and thus 
$\beta\in\overline{\R^n_{>0}\setminus S}$. 
Moreover, since $b\in \Delta_a$, $b_1\leq a_1$, so that $\beta_1=(a\vee b)_1=a_1$, and
thus $\beta\in \overline{U_a\setminus S}$. 
Assume that there is a $\gamma\in \overline{U_a\setminus S}$ such that
$\gamma\leq \beta$. Then, as above, there is an inner corner $c\neq a$ of $S$ 
such that $c\prec \gamma+(0,1,\ldots, 1)$ 
and $\gamma=a\vee c$. 
Now the inner corners $a, b, c$ of $S$ satisfy $a\vee b\vee c=a\vee
b$. Since $b\in\Delta_a$, $\{a,b\}$ is an edge of $\Delta$, which
means that $\{z^a, z^b\}$ is the unique set of minimal generators with
least common multiple $z^{a\vee b}$. It follows that $c\in\{a,b\}$ and
since $c\neq a$, we have that $c=b$, and thus $\gamma=\beta$. Hence
$\beta$ is minimal in $\overline{U_a\setminus S}$. 
We conclude that the inner corners of $\T_a$ are exactly the lattice
points $a\vee b$, where $b\in \Delta_a^*(1)$.

\end{proof}

Let $\pi: \R^n_{x_1,\ldots, x_n}\to \R^{n-1}_{x_2,\ldots, x_n}$ be the projection 
$\pi: (x_1,\ldots, x_n) \mapsto (x_2,\ldots, x_n)$ 
and let $\rho:\R^n_{x_1,\ldots, x_n}\to \R^{n-1}_{x_2,\ldots, x_n}$ be the
affine mapping 
defined by $\rho(x)=\pi(x-a)$. 
Let $M_a$ be the monomial ideal in $A_{n-1}$ (with variables $z_2,\ldots,
z_n$) 
defined by
$\T_a$ and $\rho$ as in Section ~\ref{trappsektion}.

\begin{claim}\label{generiskt} 
$M_a$ is a generic Artinian monomial ideal. The Scarf complex
$\Delta_{M_a}$ consists of faces of the form 
$\{\rho (a\vee b^1), \ldots, \rho (a\vee b^j)\}$, 
where $\{b^1,\ldots, b^j\}$ is a face of $\Delta_a^*$. 
\end{claim} 

\begin{proof}
Since $S$ is bounded, $T_a\subset S$ is bounded, and thus $M_a$ is
Artinian.

To show that $M_a$ is generic, assume that there are two minimal
generators $z^\beta$ and $z^\gamma$ that have the same positive degree
in some variable, i.e., $\beta_\ell=\gamma_\ell$ for some $2\leq \ell\leq
n$. 
Assume that $\beta=\rho(a\vee b)$ and $\gamma=\rho (a\vee c)$, where
$b,c\in\Delta_a^*(1)$.
Then $\beta_\ell=(a\vee b)_\ell-a_\ell$ and $\gamma_\ell=(a\vee
c)_\ell-a_\ell$, and thus $\beta_\ell=\gamma_\ell>0$ implies that 
$(a\vee b)_\ell=(a\vee c)_\ell >a_\ell$, which in turn implies that
$b_\ell=c_\ell$. 
Since $M$ is
generic there is a minimal generator $z^d$ of $M$ that strictly
divides $\lcm (z^b, z^c)$, i.e., $d_k <(b\vee c)_k$ for all 
$k$ such that $(b\vee c)_k>0$. In particular,
$d_1<(b\vee c)_1=a_1$. 
Set $\delta=\rho(a\vee d)$ and take $k$ such that 
\[
0<(\beta\vee \gamma)_k = \big ( (a\vee b-a)\vee (a\vee c - a)\big
)_k=(a\vee b \vee c-a)_k.
\]
Then $(b\vee c)_k>a_k\geq 0$ and thus $d_k<(b\vee c)_k$. This implies 
that 
\[
\delta_k=(d\vee a - a)_k < (a\vee b \vee c - a)_k = (\beta\vee
\gamma)_k.
\] 
It follows that $z^\delta$ strictly divides $\lcm(z^\beta, z^\gamma)=z^{\beta\vee \gamma}$. 
Now $a\vee d\in\overline{U_a\setminus S}$, since $a\vee d\geq a$,
$(a\vee d)_1=a_1$ and $z^{a\vee d}\in M$. 
Hence $z^\delta\in M_a$ by 
definition. To conclude, $M_a$ is a generic monomial ideal.

Since $M_a$ is generic, the Scarf complex $\Delta_{M_a}$ is simplicial
and the top-dimensional faces are precisely the outer corners of the
staircase of $M_a$, i.e., simplices of the form $\{\rho(a\vee b^1),\ldots, \rho(a\vee
b^{n-1})\}$, where $\{b^1,\ldots, b^{n-1}\}$ is in
$\Delta_a^*(n-1)$. Since $\Delta_{M_a}$ and $\Delta_a$ are
simplicial, it follows that the faces of $\Delta_{M_a}$ are of
the desired form. 
\end{proof}

\begin{claim}\label{disjunkta}
Assume that $a\neq b$ are inner corners of $S$. Then 
$\pi (\T_a) \cap \pi (\T_b)=\emptyset$. 
\end{claim}
In particular, it follows that 
\begin{equation}\label{nollbrutal}
\T_a\cap \T_b=\emptyset \text{ if } a\neq
b.
\end{equation}

\begin{proof} 
Let us first assume that $a_1\neq b_1$ and that $\pi (\T_a)\cap \pi
(\T_b)\neq \emptyset$; without loss of generality we may assume that
$a_1>b_1$.
Pick $\beta=(\beta_2,\ldots, \beta_n)\in \pi (\T_a)\cap \pi (\T_b)$
and let $\gamma=(a_1,\beta_2,\ldots, \beta_n)$. 
Since $\beta\in \pi(\T_b)$, $\beta_j>b_j$ for $j=2,\ldots, n$ and since
$a_1>b_1$ it follows that $\gamma \succ b$. Hence $\gamma$ is in
the interior of $\R^n_{>0}\setminus S$. On the other hand,
$\beta\in\pi(\T_a)$ implies that $\gamma\in \T_a \subset \partial
S$, which contradicts that $\gamma$ is in
the interior of $\R^n_{>0}\setminus S$. Thus $\pi (\T_a) \cap \pi (\T_b) = \emptyset$ if $a_1\neq b_1$. 

Next, assume that $a_1=b_1>0$ and that $\pi(\T_a)\cap \pi (\T_b)\neq
\emptyset$; if $a_1=0$ or $b_1=0$, the claim is trivially true by
Claim ~\ref{lila}. Since $\T_a$ and $\T_b$ are both contained in the hyperplane
$H_a$,  $\pi(\T_a)\cap \pi (\T_b)\neq
\emptyset$ is equivalent to that $\T_a\cap \T_b\neq
\emptyset$. Assume that $\beta\in \T_a\cap \T_b$. Then
$\beta_\ell>(a\vee b)_\ell$ for $\ell=2,\ldots, n$. Since $M$ is
generic and $a_1=b_1$, 
there is a minimal generator $z^c$ that strictly divides $\lcm (z^a,
z^b)$; in particular $c_1<(a\vee b)_1=\beta_1$. It follows that 
$\beta\succ c$, 
and thus $\beta$ is contained in
the interior of $\R^n_{>0}\setminus S$, which contradicts that
$\beta\in \T_a\cap \T_b\subset \partial S$. 
Thus we have proved that $\pi(\T_a)\cap \pi (\T_b)=\emptyset$ when
$a\neq b$. 
\end{proof}

\begin{claim}\label{finnsen}
For each $x\in S$, there is an inner corner $a$ of $S$ such that $\pi
(x)\in \pi (\T_a)$. 
\end{claim}

\begin{proof}

Consider $x\in\R^n_{>0}$. 
Then there is at least one inner corner $a^0$ of $S$ such that
$\pi(x)\succ \pi(a^0)$. 
Indeed, since $M$ is Artinian, there is a
generator of the form $z_1^{a_1}$, whose exponent $(a_1,0,\ldots, 0)$
is mapped to the origin in $\R^{n-1}$, and thus we can choose
$a^0=(a_1,0,\ldots, 0)$. 

Given such an $a^0$, either $\pi(x)\in\pi(\T_{a^0})$ or $\pi(x)\succ \pi(\beta^1)$ for some inner
corner $\beta^1$ of $\T_{a^0}$. In the latter case, by Claim ~\ref{lila}, $\beta^1=a^0\vee a^1$, where
$a^1\in\Delta_{a^0}^*(1)$; in particular, $a^1_1<a^0_1$. 
Now $\pi(x)\succ \pi(a^0\vee a^1)\geq \pi(a^1)$, which implies that
either $\pi(x)\in\pi(\T_{a^1})$ or $\pi(x)\succ \pi(\beta^2)$ for some
$\beta^2=a^1\vee a^2$, where $a^2\in\Delta_{a^1}^*(1)$; in particular, $a^2_1<a^1_1$.

By repeating this argument we get a sequence of inner corners
$a^0,\ldots, a^k$, such that $\pi(x)\succ \pi(a^j)$ for $j=0,\ldots, k$
and either $\pi(x)\in \pi(\T_{a^k})$ or $a^k_1=0$. 
If $a^k_1=0$, then $\pi(x)\succ \pi(a^k)$ implies that $x\succ a^k$,
which means that $x\notin S$. Hence, either $\pi(x)\in\pi(\T_a)$ for
some inner corner $a$ of $S$ or $\pi\notin S$.

\end{proof}

Next, we will use the staircases $\T_a$ to construct a partition of $S$. 
For each inner corner $a$ of $S$, let 
\[
P_a=\{x\in S\mid \pi (x)\in\pi (\T_a)\}. 
\]
In other words, $P_a$ consists of everything in $S$ ``below'' the staircase
$\T_a$. By a slight abuse of notation,  $P_a=]0,a_1]\times \T_a$. 

\begin{remark}\label{parti}
By Claim \ref{finnsen}, each $x\in S$ is contained in a $P_a$ for some
inner corner $a$, and by Claim ~\ref{disjunkta} the intersection
$P_a\cap P_b$ is empty if $a$ and $b$ are different inner
corners. Thus the set of (non-empty) $P_a$ gives a partition of $S$. 
\end{remark}

\begin{remark}\label{februari}
Note that $P_a$ is a staircase itself with the same outer corners as
$\T_a$, i.e., $\alpha\in \Delta_a(n)$.
\end{remark} 
Next, we will see that each $S_\alpha$ is contained in a
$P_a$.

\begin{claim}\label{etthorn} 
For each $\alpha\in \Delta(n)$, $S_\alpha$ is contained in a
$P_a$. More precisely, if $\alpha\in \Delta_a(n)$, then
$S_\alpha\subset P_a$. 
\end{claim}
\begin{proof}

Let us fix $\alpha\in \Delta(n)$. Recall from Section \ref{delkomplex}
that there is a unique $a$ such that $\alpha\in \Delta_a(n)$. We need
to show that $S_\alpha\cap P_b=\emptyset$ for all $b\neq a$. 

We first consider the case when $b$ is such that $b_1>\alpha_1$. Take 
$x\in P_b$. By Remark ~\ref{februari}, $P_b$ is a staircase with outer
corners $\Delta_b(n)$. It follows that $x\leq \beta$ for some
$\beta\in\Delta_b(n)$, which, by the definition of the $S_\gamma$,
implies that $x\in \bigcup_{\gamma\geq_\sigma \beta} S_\gamma$. Since
$\beta_1=b_1>\alpha_1$, $\beta>_\sigma \alpha$, and thus, since the
$S_\gamma$ are disjoint, $x\notin S_\alpha$. We conclude that
$S_\alpha\cap P_b=\emptyset$ in this case. 

Next we consider the case when $b_1=\alpha_1$. 
Assume that $x\in S_\alpha \cap P_b$. 
Then $\alpha_\ell\geq x_\ell > b_\ell$ for
$\ell=2,\ldots, n$. Since $\alpha_1=b_1$ it follows that $\alpha\in\T_b$. 
On the other hand, by Claim ~\ref{lila}, $\alpha\in \Delta_a(n)$
implies that $\alpha\in \T_a$, which, by
\eqref{nollbrutal}, contradicts that $\alpha\in\T_b$. 
It follows that 
$S_\alpha\cap P_b=\emptyset$. 

Finally we consider the case when $b_1<\alpha_1$. Assume that $x\in
S_\alpha\cap P_b$. Then, as above, $\alpha_\ell\geq x_\ell> b_\ell$ for
$\ell=2,\ldots, n$. Since also $\alpha_1>b_1$, it follows that
$\alpha\succ b$, which however contradicts that $b$ is an inner corner
of $S$. Hence $S_\alpha\cap P_b=\emptyset$ also in this case, which
concludes the proof.

\end{proof}

\medskip

Next, we will inductively define staircases and partitions of $S$ associated
with faces of 
$\Delta$ of higher dimension. 
Given vertices $a^1,\ldots, a^{k-1}$ of $\Delta$, such that
$\Delta_{a^1,\ldots, a^{k-1}}$ is non-empty (in particular,  $a^j$ is
in $\Delta_{a^1,\ldots, a^{j-1}}$ for $j=2,\ldots, k-1$) and 
an inner corner $a^k$ of $\Delta_{a^1,\ldots, a^{k-1}}^*$, assuming
that $\T_{a^1,\ldots, a^{k-1}}$ is defined, we let 
\[
\T_{a^1,\ldots, a^k}:=\{x\in \T_{a^1,\ldots, a^{k-1}} \mid
x_k=(a^1\vee\cdots\vee a^k)_k, x_j > (a^1\vee\cdots\vee a^k)_j,
j=k+1,\ldots, n\}. 
\]
Recall that by the definition of the sequence $a^1,\ldots, a^k$, in
fact, $(a^1\vee\cdots\vee a^k)_j=a_j^j$ for $j=1,\ldots, k$, see Section \ref{delkomplex}. 
Moreover, note that $\T_{a^1,\ldots, a^k}$ is contained in the
codimension $k$-plane 
\[
H_{a^1,\ldots, a^k}:=\{x_1=(a^1\vee\cdots \vee a^k)_1,  \ldots, x_k=(a^1\vee\cdots\vee
a^k)_k\}=\{x_1=a^1_1, \ldots, x_k=a_k^k\}.
\]

Let 
$\pi_k:\R^n_{x_1,\ldots, x_n}\to \R^{n-k}_{x_{k+1},\ldots, x_n}$ be the projection
$\pi_k:(x_1,\ldots, x_n)\mapsto (x_{k+1},\ldots, x_n)$, 
and let $\rho_k: \R^n_{x_1,\ldots, x_n}\to \R^{n-k}_{x_{k+1},\ldots, x_n}$ be the
affine mapping 
defined by $\rho_k: x\mapsto \pi_k(x-a^1\vee\cdots\vee a^k)$.

\begin{claim}\label{fredag} 
Assume that $\T_{a^1,\ldots, a^{k-1}}$ is non-empty and that
$a_k\in\Delta_{a^1,\ldots, a^{k-1}}^*(1)$. Then $\T_{a^1,\ldots,
  a^k}=\emptyset$ if and only if $a_k^k=(a^1\vee\cdots\vee
a^{k-1})_k$. 
If $\T_{a^1,\ldots, a^k}$ is non-empty, then it is a staircase 
in $H_{a^1,\ldots, a^k}$. 
The origin of $\T_{a^1,\ldots, a^k}$ is $a^1\vee\cdots\vee a^k$, the
outer corners are the top-dimensional faces of $\Delta_{a^1,\ldots,
  a^{k}}$ and the inner corners are the lattice points of the form 
$a^{1}\vee\ldots\vee a^{k}\vee b$, where $b$ is a vertex of
$\Delta_{a^1,\ldots, a^{k}}^*$. 

The monomial ideal $M_{a^1,\ldots, a^k}$ defined by $\T_{a^1,\ldots,
  a^k}$ and $\rho_k$ as in Section ~\ref{trappsektion}, i.e., it has 
staircase $\rho_k(\T_{a^1,\ldots, a^k})$, is an Artinian generic
monomial ideal. The Scarf complex $\Delta_{M_{a^1,\ldots, a^k}}$
consists of faces of the form 
$\{\rho_k(a^1\vee\cdots\vee a^k\vee b^1), \ldots,
\rho_k(a^1\vee\cdots\vee a^k\vee b^j)\}$,
where $\{b^1,\ldots, b^j\}$ is a face of $\Delta_{a^1,\ldots,
  a^k}^*$. 
\end{claim}

\begin{proof} 
By Claims ~\ref{lila} and ~\ref{generiskt} the claim holds for
$k=1$. Assume that it holds for $k=\k-1$; we then need to prove that it holds for $k=\k$. 

\smallskip

First, it is clear from the definition that $\T_{a^1,\ldots, a^\k}$ is
contained in $H_{a^1,\ldots, a^\k}$. 
If $a_\k^\k=(a^1\vee\cdots\vee a^{\k-1})_\k$ then $H_{a^1,\ldots, a^\k}\cap
\T_{a^1,\ldots, a^{\k-1}}=\emptyset$ and thus $T_{a^1,\ldots, a^\k}$ is
empty. 
In general, note that $T_{a^1,\ldots, a^\k}$ is empty exactly if
$a^1\vee\cdots\vee a^\k+(0,\ldots, 0,1,\ldots, 1)\notin S$; here
$(0,\ldots, 0,1,\ldots, 1)$ means that the first $\k$ entries are $0$
and the rest are $1$. 
Assume that $a_\k^\k\neq (a^1\vee\cdots\vee a^{\k-1})_\k$. Then,
by thef definition of the $a_j^j$, in fact, $a_\k^\k>(a^1\vee\cdots\vee a^{\k-1})_\k\geq 0$. Since
$T_{a^1,\ldots, a^{\k-1}}\neq \emptyset$ and the claim holds for $k=\k-1$
by assumption, $a_j^j>0$ for $j=1,\ldots, \k-1$, and thus $a+(0,\ldots,
0,1,\ldots, 1)\in \R^n_{>0}$. Then the condition 
$a^1\vee\cdots \vee a^\k+(0,\ldots, 0,1,\ldots, 1)\notin S$ implies
that there is an inner corner $c$ of $S$ such that $c\prec
a+(0,\ldots, 0,1,\ldots, 1)$. In particular, $c\leq a$, which
contradicts that $a$ is an inner corner. Thus $T_{a^1,\ldots, a^\k}\neq\emptyset$ if
$a_\k^\k\neq (a^1\vee\cdots\vee a^{\k-1})_\k$.

\smallskip

Let us now assume that $a_\k^\k>(a^1\vee\cdots\vee a^{\k-1})_\k$. 
We will use Claim ~\ref{lila} to show that 
$T_{a^1,\ldots, a^\k}$ 
 is a
staircase of the desired form. 
Let
$\widetilde S$ be the staircase $\rho_{\k-1}(T_{a^1,\ldots,
  a^{\k-1}})\subset \R^{n-\k+1}_{x_\k,\ldots, x_n}$ 
  of $M_{a^1,\ldots, a^{\k-1}}$ and choose an inner corner $\tilde a:=
  \rho_{\k-1}(a^1\vee\cdots\vee a^\k)$ of $\widetilde S$. 
Then note that 
\begin{multline}\label{fiol}
\T_{\tilde{a}}=\{x\in \widetilde S\mid x_\k=\tilde a _\k,
x_\ell> \tilde a_\ell, \ell=\k+1,\ldots, n\} 
=\\
\{x\in\rho_{\k-1}(\T_{a^1,\ldots, a^{\k-1}})\mid x_\k = 
\rho_{\k-1}(a^1\vee \cdots \vee a^\k)_\k,  
x_\ell > \rho_{\k-1}(a^1\vee \cdots \vee a^\k)_\ell, \ell=\k+1,\ldots,
n\}
=\\
\rho_{\k-1}(\T_{a^1,\ldots, a^\k}).
\end{multline}
Now, by Claim ~\ref{lila}, $T_{\tilde a}$ is a staircase in the
hyperplane $\{x_\k=\tilde a_\k\}\subset\R^{n-\k+1}_{x_\k,\ldots, x_n}$
with origin ~$\tilde a$. 
The outer
corners are the top-dimensional faces
$\tilde\alpha$ of $\Delta_{\tilde a}$, i.e., the top-dimensional faces
$\tilde\alpha$ of $\Delta_{M_{a^1,\ldots, a^{\k-1}}}$ such that 
\begin{equation}\label{nodnod}
\tilde\alpha_\k=\tilde a_\k=\rho_{\k-1}(a^1\vee\cdots \vee a^\k)_\k=(a^1\vee\cdots \vee
a^\k)_\k=a^\k_\k.
\end{equation}
Since the
  claim holds for $k=\k-1$, that $\tilde\alpha$ is a top-dimensional face
  of $\Delta_{M_{a^1,\ldots,
      a^{\k-1}}}$ means that
    $\tilde\alpha=\rho_{\k-1}(a^1\vee\cdots\vee a^{\k-1}\vee \beta)$,
    where $\beta$ is a top-dimensional face of $\Delta_{a^1,\ldots,
      a^{\k-1}}^*$. In other words, $\tilde\alpha=\rho_{\k-1}(\alpha)$,
    where $\alpha$ is a top-dimensional face of $\Delta$ such that
    $\alpha_j=a^j_j$ for $j=1,\ldots, \k-1$. By
    \eqref{nodnod} we also have that
    $\alpha_\k=\rho_{\k-1}(\alpha)_\k=a_\k^\k$, so that $\alpha\in \Delta_{a^1,\ldots, a^\k}(n)$. To conclude, 
    the outer corners of $T_{\tilde a}$ are of the form
    $\rho_{\k-1}(\alpha)$ where $\alpha\in\Delta_{a^1,\ldots, a^\k}(n)$.

Moreover, by Claim \ref{lila} the inner corners of $\T_{\tilde a}$ are the lattice points $\tilde a\vee \tilde b$,
where $\tilde b$ is a vertex of $\Delta_{\tilde a}^*$.
Since the lemma holds for $k=\k-1$, this means
that $\tilde b=\rho_{\k-1}(a^1\vee\cdots\vee a^{\k-1}\vee b)$, where $b\in\Delta_{a^1,\ldots, a^{\k-1}}^*(1)$.
Since $\tilde b\neq \tilde a$, $b\neq a^\k$, and thus 
$\tilde a\vee \tilde b=\rho_{\k-1}(a^1\vee \cdots \vee a^\k\vee b)$,
where $b\in \Delta_{a^1,\ldots, a^\k}^*(1)$.

Since the restriction $\rho_{\k-1}:H_{a^1,\ldots, a^{\k-1}}\to
\R^{n-\k+1}$ 
is a just a translation of the plane $H_{a^1,\ldots, a^{\k-1}}$
(if we consider $\R^{n-\k+1}$ as
embedded in $\R^n$) 
it follows that 
$\T_{a^1,\ldots, a^\k}$ is a
staircase in $H_{a^1,\ldots, a^\k}$ with origin $a^1\vee\cdots\vee a^\k$, where the
outer corners are the top-dimensional faces of $\Delta_{a^1,\ldots,
  a^\k}$ and the inner corners are of the form $a^1\vee \cdots \vee
a^\k\vee b$, where $b\in \Delta_{a^1,\ldots, a^\k}^*(1)$. This proves
the first part of the claim.

\smallskip 

Next, we will use Claim ~\ref{generiskt} to prove the second part of
the claim. 
Let $\tilde\rho:\R^{n-\k+1}_{x_\k,\ldots, x_n}\to \R^{n-\k}_{x_{\k+1},\ldots,
  x_n}$ be the affine map $\tilde\rho:(x_{\k},\ldots, x_n)\mapsto
(x_{\k+1}-\tilde a_{\k+1},\ldots, x_n-\tilde a_n)$. 
Note that 
$\rho_\k=\tilde\rho \rho_{\k-1}$. It follows that the ideal $M_{a^1,\ldots, a^\k}$ 
has staircase 
\[
\rho_\k(T_{a^1,\ldots, a^\k})=\tilde\rho \rho_{\k-1}(T_{a^1,\ldots,
  a^\k})= \tilde\rho(T_{\tilde a}), 
\]
where we have used \eqref{fiol} for the second equality. 
In other words, $M_{a^1,\ldots, a^\k}$ 
is the
ideal defined by $\T_{\tilde a}$ and $\tilde\rho$ as in Section
~\ref{trappsektion}. 
Thus by Claim
~\ref{generiskt}, it is an Artinian generic monomial ideal.

Moreover, by Claim \ref{generiskt}, the Scarf complex $\Delta_{M_{a^1,\ldots, a^\k}}$ consists of
faces of the form 
\begin{equation}\label{invasion}
\{\tilde \rho(\tilde a\vee \tilde b^1), \ldots, \tilde \rho(\tilde
a\vee \tilde b^j)\}, 
\end{equation}
where $\{\tilde b^1,\ldots, \tilde b^j\}$ is a
face of $\Delta_{\tilde a}^*$.
As above, 
 $\tilde b^\ell\in \Delta_{\tilde
  a}^*(1)$ implies that $\tilde a\vee \tilde b^\ell=\rho_{\k-1}(a^1\vee\cdots\vee a^{\k-1}\vee
b^\ell)$, where $b\in\Delta_{a^1,\ldots, a^\k}^*(1)$. 
Hence, 
\[
\tilde \rho(\tilde a\vee \tilde b^\ell)=
\tilde\rho\rho_{\k-1}(a^1\vee\cdots\vee
a^\k\vee b^\ell)=
\rho_\k(a^1\vee\cdots\vee
a^\k\vee b^\ell)
\]
where $b\in\Delta_{a^1,\ldots, a^\k}^*(1)$. Thus the faces
\eqref{invasion} are of the desired form, and we have proved the
second part of the claim.

\end{proof}

To construct the partitions associated with the staircases
$\T_{a^1,\ldots, a^k}$, we 
define inductively 
\[
P_{a^1,\ldots, a^k}=\{x\in P_{a^1,\ldots, a^{k-1}}\mid \pi_k(x)\in
\pi_k(\T_{a^1,\ldots, a^k})\}. 
\]
Then $P_{a^1,\ldots, a^k}$ is a $k$-dimensional cuboid times the 
$(n-k)$-dimensional staircase $\T_{a^1,\ldots, a^k}$. The $\ell$th side length is
given as the ``height'' of $\T_{a^1,\ldots, a^\ell}$ in $\T_{a^1,\ldots,
  a^{\ell-1}}$, which equals $(a^1\vee\cdots\vee
a^\ell-a^1\vee\cdots\vee a^{\ell-1})_\ell$.   
By a slight abuse of notation 
\begin{equation*}
P_{a^1,\ldots, a^k}=
]0,a^1_1]\times ]a^1_2, (a^1\vee a^2)_2]\times 
\cdots \times 
](a^1\vee \cdots \vee a^{k-1})_k, (a^1\vee \cdots \vee a^k)_k]\times
\T_{a^1,\ldots, a^k}.
\end{equation*}
In particular, 
\begin{equation}\label{pack}
P_{a^1,\ldots, a^n}=
]0,a^1_1]\times ]a^1_2, (a^1\vee a^2)_2]\times \cdots \times 
](a^1\vee \cdots \vee a^{n-1})_n, (a^1\vee \cdots \vee a^n)_n].
\end{equation} 
\begin{remark}\label{bolla}
Note that $P_{a^1,\ldots, a^k}$ is, in fact, a staircase with 
outer corners 
$\alpha\in\Delta_{a^1,\ldots, a^k}(n)$.
\end{remark}

\begin{claim}\label{extraclaim}
For each $k$, the set of non-empty $P_{a^1,\ldots, a^k}$ gives a
partition of $S$. 
\end{claim}
In particular, since $\T_{a^1,\ldots, a^k}\subset P_{a^1,\ldots, a^k}$, 
\begin{equation}\label{brutal}
\T_{a^1,\ldots, a^k}\cap \T_{b^1,\ldots, b^k}=\emptyset \text{ if }
\{a^1,\ldots, a^k\}\neq \{b^1,\ldots, b^k\}.
\end{equation}

\begin{proof}

By Remark \ref{parti}, the claim holds for
$k=1$. 
Assume that the claim holds for $k=\k-1$. 
To prove that it holds for $k=\k$ it suffices to show that, given a
sequence $a^1,\ldots, a^{\k-1}$ of inner corners, the set of
$P_{a^1,\ldots, a^\k}$, where $a^\k\in\Delta_{a^1,\ldots,
  a^{\k-1}}^*(1)$, gives a partition of $P_{a^1,\ldots, a^{\k-1}}$. Take $x\in P_{a^1,\ldots,
  a^{\k-1}}$. 
We then need to show that there is
  exactly one choice of $a^\k\in\Delta_{a^1,\ldots, a^{\k-1}}^*(1)$ such
  that $\pi_\k(x)\in\pi_\k(T_{a^1,\ldots,
    a^\k})$.

Let $\widetilde S$ be the staircase $\pi_{\k-1}(T_{a^1,\ldots,
  a^{\k-1}})\subset \R^{n-\k+1}_{x_\k,\ldots, x_n}$. By Claim
  \ref{fredag}, $\widetilde S$ is a staircase with (origin
  $\pi_{\k-1}(a^1\vee\cdots\vee a^{\k-1})$ and) inner corners
  $\pi_{\k-1}(a^1\vee\cdots\vee a^\k)$, where $a^\k\in\Delta_{a^1,\ldots,
    a^{\k-1}}^*(1)$. 
Given an inner corner $\tilde a:=\pi_{\k-1}(a^1\vee\cdots\vee a^\k)$ of
$\widetilde S$, 
note that 
\begin{multline}\label{viola}
\T_{\tilde a}=\{x\in \widetilde S\mid x_\k=\tilde a_\k, x_\ell > \tilde
a_\ell \text{ for } \ell=\k+1,\ldots, n\}=\\
\{x\in\pi_{\k-1}(\T_{a^1,\ldots, a^{\k-1}})\mid x_\k=
 \pi_{\k-1}(a^1\vee\cdots\vee a^\k)_\k, 
x_\ell> \pi_{\k-1}(a^1\vee\cdots\vee a^\k)_\ell \text{ for } \ell=\k+1,\ldots, n\}=\\
\pi_{\k-1}(\T_{a^1,\ldots, a^\k})
\end{multline}
Let $\tilde\pi:\R^{n-\k+1}_{x_\k,\ldots, x_n}\to
\R^{n-\k}_{x_{\k+1},\ldots, x_n}$ be the projection $(x_\k,\ldots,
x_n)\mapsto (x_{\k+1},\ldots, x_n)$. Then, clearly, $\tilde\pi
\pi_{\k-1}=\pi_\k$.  
By a slight modification of the proofs, Claims ~\ref{disjunkta}
and ~\ref{finnsen} hold also for staircases with origin different from
$0$; it follows that 
for $\tilde x\in \widetilde S$ there is exactly one inner corner
$\tilde a=\pi_{\k-1}(a^1\vee\cdots\vee a^\k)$ of $\widetilde S$ such
that 
\[
\tilde \pi
(\tilde x)\in \tilde \pi (\T_{\tilde a})=\tilde \pi
\pi_{\k-1}(\T_{a^1,\ldots, a^\k})=\pi_\k(\T_{a^1,\ldots, a^\k}),
\]
 where we
have used \eqref{viola} for the second equality. 
Now take $x\in P_{a^1,\ldots, a^{\k-1}}$. Then
$\pi_{\k-1}(x)\in\pi_{\k-1}(T_{a^1,\ldots, a^{\k-1}})$, and thus there is
  exactly one choice of $a^\k\in\Delta_{a^1,\ldots, a^{\k-1}}^*(1)$ such
  that 
$\tilde\pi \big (\pi_{\k-1}(x)\big ) =\pi_\k(x)\in\pi_\k(T_{a^1,\ldots,
  a^\k})$. This concludes the proof.

\end{proof}

\begin{claim}\label{flerahorn} 
For each $k$ and each $\alpha\in \Delta(n)$, there is a unique
sequence of $k$ inner corners 
$a^1,\ldots, a^k$ such that $S_\alpha$ is contained in $P_{a^1,\ldots,
  a^k}$. 
More precisely, if  $\alpha\in \Delta_{a^1,\ldots, a^k}$,
  then $S_\alpha\subset P_{a^1,\ldots, a^k}$. 
\end{claim} 

\begin{proof} 

Let us fix $k$. 
Recall from Section \ref{delkomplex} that given $\alpha\in
\Delta(n)$, there is a unique sequence of $k$ inner corners
$a^1,\ldots, a^k$ such that $\alpha\in \Delta_{a^1,\ldots,
  a^k}(n)$. Also, recall from the definition of $\Delta_{a^1,\ldots,
  a^k}$ that $(a^1\vee\cdots \vee a^k)_j=a^j_j$ for $j=1,\ldots, k$. 
To prove the claim we need to show that $S_\alpha\cap P_{b^1,\ldots,
  b^k}=\emptyset$ for all sequences of $k$ inner corners $b^1,\ldots, b^k$ different from
$a^1,\ldots, a^k$. 

We first consider the case when there is an $\ell\leq k$, such that $b^j_j=a^j_j$
for $j<\ell$ and $b_\ell^\ell>a_\ell^\ell$. Pick $x\in P_{b^1,\ldots,
  b^k}$. Since $P_{b^1,\ldots, b^k}$ is a staircase with outer corners
in $\Delta_{b^1,\ldots, b^k}(n)$, see Remark \ref{bolla}, it follows that $x\leq \beta$ for
some $\beta\in\Delta_{b^1,\ldots, b^k}(n)$. 
By the definition of the $S_\gamma$, then $x\in
\bigcup_{\gamma\geq_\sigma \beta} S_\gamma$. 
Since $\beta^j_j=b^j_j=a^j_j$ for $j=1,\ldots, \ell-1$ and
$\beta^\ell_\ell=b^\ell_\ell>a^\ell_\ell$, $\beta >_\sigma \alpha$, and
thus, since the $S_\gamma$ are disjoint, $x\notin S_\alpha$. We
conclude that $S_\alpha\cap P_{b^1,\ldots,
  b^k}=\emptyset$ in this case. 

Next, we consider the case when $b^j_j=a^j_j$ for $j=1,\ldots,
k$. Assume that $x\in S_\alpha\cap P_{b^1,\ldots,
  b^k}$. Then $\alpha_j=a^j_j=b^j_j$ for $j=1,\ldots, k$ and
$\alpha_j\geq x_j>(b^1\vee\cdots\vee b^k)_j$ for $j=k+1,\ldots,
n$. Thus by definition $\alpha\in T_{b^1,\ldots, b^k}$. 
On the other hand, by Claim \ref{fredag} $\alpha\in T_{a^1,\ldots,
  a^k}$, which by \eqref{brutal} contradicts that $\alpha\in
T_{b^1,\ldots, b^k}$. 
It follows that $S_\alpha\cap P_{b^1,\ldots,
  b^k}=\emptyset$. 

Finally we consider the case when there is an $\ell\leq k$, such that
$b_j^j=a_j^j$ for $j<\ell$ and $b_\ell^\ell<a_\ell^\ell$. If $\ell=1$
we know from (the proof of) Claim \ref{etthorn} that $S_\alpha\cap
P_{b^1}=\emptyset$ and thus 
$S_\alpha\cap P_{b^1,\ldots,
  b^k}\subset S_\alpha\cap P_{b^1}=\emptyset$. 
Assume that $\ell\geq 2$ and that 
\[
x\in S_\alpha\cap P_{b^1,\ldots,
  b^k}\subset S_\alpha\cap P_{b^1,\ldots, b^\ell}.
\]
Then $\alpha_j=a_j^j=b_j^j$ for
$j=1,\ldots, \ell-1$, $\alpha_\ell=a^\ell_\ell > b_\ell^\ell$, and $\alpha_j\geq x_j > (b^1\vee \cdots \vee b^\ell)_j$ for
$j=\ell+1,\ldots, n$. 
It follows that 
$\alpha\in T_{b^1,\ldots, b^{\ell-1}}$, but, from Claim \ref{fredag}
we know that 
$\alpha\in T_{a^1,\ldots, a^{\ell-1}}$, which 
leads to a contradiction by \eqref{brutal}. 
Hence $S_\alpha\cap P_{b^1,\ldots,
  b^k}=\emptyset$ also in this case. 

\end{proof}

Recall from Section ~\ref{delkomplex} that $\Delta_{a^1,\ldots, a^n}$
is just the simplex $\alpha\in\Delta(n)$ with vertices $a^1,\ldots,
a^n$. 
On the other hand, 
each outer corner $\alpha$ gives rise to a non-empty
$P_\alpha:=P_{a^1,\ldots, a^n}$ by choosing $a^\ell$ 
 as the $x_\ell$-vertex of $\alpha$. 
By Claim ~\ref{flerahorn}, $S_\alpha\subset P_\alpha$, and since both $\{P_\alpha\}$ and $\{S_\alpha\}$
give partitions of $S$, we conclude that $S_\alpha=P_\alpha$. 

Now, given ${\I}=\{i_1,\ldots, i_n\}\in \Delta (n)$,  
we choose $a^\ell$ as the $x_\ell$-vertex 
$a^{i_\eta(\ell)}$. Then Lemma 
~\ref{ratblock} follows in light of \eqref{pack}.

\smallskip 

For a general choice of $\sigma$ the above proof works verbatim, with
the 
coordinates $x_\ell$ and the variables $z_\ell$ replaced by $x_{\sigma(\ell)}$ and
$z_{\sigma(\ell)}$, respectively, and $\eta$ replaced by $\tau$.

\subsection{Computing $d\varphi$}\label{raknautdf}

Let us now compute the $e_{\I}^*$-entry of $d_\sigma\varphi$ for a given
${\I}=\{i_1,\ldots, i_n\}\in \Delta(n)$. 
Recall from \eqref{blabla} that 
\begin{equation}
\varphi_k = \sum_{{\mathcal J}=\{j_1, \ldots, j_k\} \subset {\I}}   \sum_{\ell=1}^k
(-1)^{\ell-1} z^{a^{j_1}\vee \cdots \vee a^{j_k}-a^{j_1}\vee \cdots
  \widehat {a^{j_\ell}} \cdots \vee a^{j_{k}}} e_{{\mathcal
    J}}^*\otimes e_{{\mathcal J}_\ell} + \varphi_k'
\end{equation}
where $\varphi_k'$ are the remaining terms that will not
contribute to the $e_{\I}^*$-entry. 
It follows that the coefficient of $e_{\I}^*$ in $d_\sigma\varphi$ equals 
\begin{multline}\label{kivas}
\sgn \big ((n,\ldots, 1)\big )\sum_\tau 
\sgn(\tau) 
\frac{\partial}{\partial z_{\sigma(1)}} 
z^{a^{i_{\tau(1)}}} dz_{\sigma(1)} \wedge 
\frac{\partial}{\partial z_{\sigma(2)}} z^{a^{i_{\tau(1)}}\vee
  a^{i_{\tau(2)}}-a^{i_{\tau(1)}}} dz_{\sigma(2)}\wedge \\
\cdots \wedge 
\frac{\partial}{\partial z_{\sigma(n-1)}} 
z^{a^{i_{\tau(1)}}\vee \cdots\vee a^{i_{\tau(n-1)}}-a^{i_{\tau(1)}}\vee
  \cdots\vee a^{i_{\tau(n-2)}}} dz_{\sigma(n-1)} 
\wedge \\
\frac{\partial}{\partial z_{\sigma(n)}}
z^{a^{i_{\tau(1)}}\vee \cdots\vee a^{i_{\tau(n)}}-a^{i_{\tau(1)}}\vee
  \cdots\vee a^{i_{\tau(n-1)}}} dz_{\sigma(n)} 
 =: \sum_\tau F_\tau, 
\end{multline}
where the sum is over all permutations $\tau$ of $\{1, \ldots, n\}$ 
and $\sgn(\tau)$ denotes the sign of the permutation $\tau$.

Let $\eta$ be the permutation of $\{1,\ldots, n\}$ associated with $\I$ as in Section
~\ref{strut} and let $\alpha$ be the label of $\I$. Then, by the
definition of $\eta$,  
$\big ( a^{i_{\tau(1)}}\vee \cdots\vee a^{i_{\tau(\kappa)}}\big
  )_\ell=\alpha_\ell$
precisely for $\ell=\eta^{-1}\big (\tau(1)\big), \ldots,
\eta^{-1}\big (\tau(\kappa)\big )$. 
It follows that 
\[z^{a^{i_{\tau(1)}}\vee \cdots\vee a^{i_{\tau(k)}}-a^{i_{\tau(1)}}\vee
  \cdots\vee a^{i_{\tau(k-1)}}}\] is a monomial in the variables
$z_{\eta^{-1}(\tau(k) )}, \ldots, z_{\eta^{-1}(\tau(n))}$. 
Therefore the last factor in $F_\tau$ vanishes unless
$\tau(n)=\eta\big (\sigma(n)\big )$. Given, $\tau(n)=\eta\big
(\sigma(n)\big )$, the next
to last factor
vanishes unless $\tau(n-1)=\eta\big (\sigma(n-1)\big)$, etc. To conclude,
$F_ \tau$, where $\tau={\eta\circ\sigma}$, is the only non-vanishing
term in \eqref{kivas}.

Now with $\tau={\eta\circ\sigma}$, 
\begin{multline}
F_\tau=
\sgn(\tau) \times 
a^{i_{\tau(1)}}_{\sigma(1)}\times \Big (a^{i_{\tau(1)}}\vee a^{i_{\tau(2)}}-a^{i_{\tau(1)}} \Big
)_{\sigma(2)}  \times \cdots \times \\
\Big ( a^{i_{\tau(1)}}\vee \cdots\vee a^{i_{\tau(n)}}-a^{i_{\tau(1)}}\vee
\cdots\vee a^{i_{\tau(n-1)}}\Big )_{\sigma(n)} 
 z^{a^{i_{\tau(1)}}\vee \cdots\vee
  a^{i_{\tau(n)}}} \frac{d z_{\sigma(n)}}{z_{\sigma(n)}}\wedge\cdots\wedge \frac{d
  z_{\sigma(1)}}{z_{\sigma(1)}}= \\
\sgn (\eta) \Vol (S_{\sigma, \alpha}) z^{\alpha-\1} dz_n\wedge\cdots\wedge dz_1, \end{multline}
where the last equality follows from Lemma ~\ref{ratblock}. This  
concludes the proof of Theorem ~\ref{main}, since, by definition
$\sgn(\alpha)=\sgn(\eta)$, see Section ~\ref{strut}.

\section{Examples}\label{exempel} 

Let us illustrate (the proof of) Theorem ~\ref{main} by some examples.

\begin{ex}\label{dimett}
Assume that $n=1$. Then each monomial ideal $M$ is a principal ideal
generated by a monomial $z^a$. The staircase of $M$ is
just the line segment $]0,a]\subset \R_{>0}$ with one outer corner
$\alpha=a$ so that $S_\alpha=S$. 
Moreover, the Scarf complex is just a point with label $z^a$, and thus
the Scarf resolution is just $0\to A\stackrel{z^a}{\longrightarrow}
A$. 
Thus in this case Theorem ~\ref{main} just reads 
\[
d\varphi = d (z^a) = \Vol (]0,a]) z^{a-1} dz = a  z^{a-1} dz. 
\]
\end{ex}

\begin{ex}\label{dimtva}
Assume that $n=2$. Then each Artinian monomial ideal $M\subset A_2$ is of the
form $M=(z_1^{a_1}z_2^{b_1}, \ldots, z_1^{a_r}z_2^{b_r})$ for some integers 
$a_1>\ldots >a_r=0$ and $0=b_1<\ldots <b_r$. 
Since no two minimal monomial generators have the same positive degree in any
variable, $M$ is trivially generic. 
In this case the staircase of $M$ looks like an actual staircase with
$r$ 
inner corners $(a_j,b_j)$ and $r-1$ outer corners $\alpha^j:=(a_j,
b_{j+1})$, 
see Figure ~5.1. In particular, $M=\bigcap_{j=1}^{r-1} (z_1^{a_j},
z_2^{b_{j+1}})$. 
\begin{figure}\label{trappor}
\begin{center}
\includegraphics{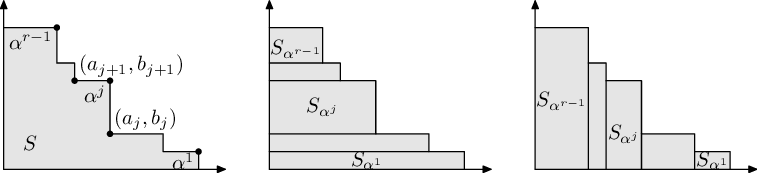}
\caption{The staircase $S$ of $M$ in Example ~\ref{dimtva} and the partitions $\{S_{\alpha^j}\}_j=\{S_{\sigma,\alpha^j}\}_j$ of $S$ corresponding to the
  permutations $\sigma=(1,2)$ and $\sigma=(2,1)$, respectively.}
\end{center}
\end{figure}
With $\sigma$ as the identity, $\T_{(a_j,b_j)}$
is just the line segment $\{x_1=a_j, b_j<x_2 \leq
b_{j+1}\}$.

Note that $\sigma=(1,2)$ corresponds to the ordering
$\alpha^1\geq_\sigma\ldots\geq_\sigma \alpha^{r-1}$ of the outer corners and 
\[
S_{(1,2), \alpha^j}=\{x\in \R^2_{>0} \mid 0 \leq x_1 < a_j, b_j \leq x_2 <
b_{j+1} \}, 
\]
whereas $\sigma=(2,1)$ corresponds to the reverse ordering of
the outer corners and so 
\[
S_{(2,1), \alpha^j}
= \{x\in\R^2_{>0}
\mid a_{j+1} \leq x_1 < a_j, 0 \leq x_2 < b_{j+1} \},
\]
see Figure ~5.1. 
Thus, the partitions just correspond to 
vertical and horisontal, respectively,  slicing of $S$.

In this case the Scarf complex is just a triangulation of the
one-dimensional simplex, and it is not very
hard to directly compute $d\varphi$, cf.\ \cite[Section~7]{LW}.  
\end{ex}

\begin{ex}\label{genex}
Let $M$ be the generic monomial ideal $M=(z_1^3, z_1^2z_2,
z_1z_2^2z_3^2, z_2^4, z_2^3z_3, z_3^3)\subset A_3$. The staircase $S$ of $M$, 
depicted in Figure ~5.2, has six inner corners, 
$a^1=(3,0,0)$, $a^2=(2,1,0)$, $a^3=(1,2,2)$, $a^4=(0,4,0)$,
$a^5=(0,3,1)$, and $a^6=(0,0,3)$, and 
five outer corners, 
$\alpha^1=(3,1,3)$, $\alpha^2=(2,4,1)$, $\alpha^3=(2,3,2)$,
$\alpha^4=(2,2,3)$, and $\alpha^5=(1,3,3)$. 
\begin{figure}\label{genfigur}
\begin{center}
\includegraphics{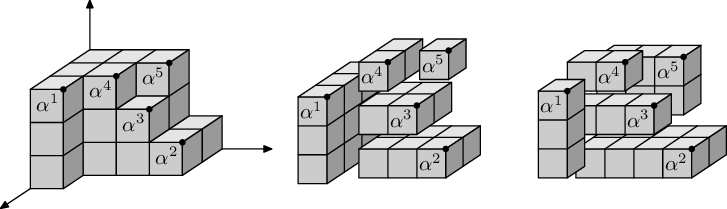}
\caption{The staircase of $M$ 
in Example ~\ref{genex} and the 
  partitions $\{S_{\alpha^j}\}_j$ corresponding to the
  orderings $\alpha^1,\alpha^2,\alpha^3,\alpha^4,\alpha^5$ and $\alpha^2,
\alpha^5, \alpha^3, \alpha^4, \alpha^1$, respectively.} 
\end{center}
\end{figure}
By Claim ~\ref{lila}, $\T_{a^j}$ is non-empty for $j=1,2,3$. These
two-dimensional staircases are the light grey regions facing the
reader in the first figure in Figure \ref{genfigur}. 

The six different permutations $\sigma$ of $\{1,2,3\}$ give rise to six different
orderings of the $\alpha^j$: for example
$\sigma^1:=(1,2,3)$ and $\sigma^2:=(2,3,1)$
correspond to the orderings
$\alpha^1\geq_\sigma\alpha^2\geq_\sigma\alpha^3\geq_\sigma\alpha^4\geq_\sigma\alpha^5$ and $\alpha^2\geq_\sigma
\alpha^5\geq_\sigma \alpha^3\geq_\sigma \alpha^4\geq_\sigma \alpha^1$, respectively. 
In the first case $S_{\sigma^1,\alpha^1}$ is the cuboid $]0,3]\times]0,1]\times]0,3]$, 
$S_{\sigma^1,\alpha^2}=]0,2]\times]1,4]\times]0,1]$,
$S_{\sigma^1,\alpha^3}=]0,2]\times]1,3]\times]1,2]$, 
$S_{\sigma^1,\alpha^4}=]0,2]\times]1,2]\times]2,3]$, and 
$S_{\sigma^1,\alpha^5}=]0,1]\times]2,3]\times]2,3]$, see Figure ~5.2,
where also the $S_{\sigma^2,\alpha^j}$ are depicted.

\end{ex}

\section{General (monomial) ideals}\label{general} 

The Scarf resolution is an instance of a more general construction of
so-called \emph{cellular resolutions} of monomial ideals, introduced
by Bayer-Sturmfels \cite{BaS}. 
The Scarf complex is then replaced by a more general oriented polyhedral
cell complex $X$, with vertices corresponding to and labeled by the  
generators of the monomial ideal $M$; as above a face $\gamma$ of $X$
is labeled by the least common multiple $m_\gamma$ of the vertices. 
Analogously to the Scarf complex, $X$ encodes a graded complex of free
$A$-modules: 
for $k=0,\ldots, \dim X+1$, let $E_k$ be a free
$A$-module 
of rank equal to the number of $(k-1)$-dimensional faces of
$X$ 
and let $\varphi_k:E_k\to E_{k-1}$ be defined by 
$\varphi_k: 
e_\gamma\mapsto \sum_{\delta\subset \gamma} \sgn(\delta,\gamma)~\frac{m_\gamma}{m_{\delta}}e_{\delta},$
where $\gamma$ and $\delta$ are faces of $X$ of dimension $k-1$ and
$k-2$, respectively, and where $\sgn(\delta,\gamma)=\pm 1$ comes from
the orientation of $X$. 
The complex $E_\bullet, \varphi_\bullet$ is exact if $X$ satisfies a
certain acyclicity 
condition see, e.g., \cite[Proposition~4.5]{MS}, and thus gives a
resolution - a so-called \emph{cellular resolution} - of the cokernel
of $\varphi_0$, which, with the
identification $E_{0}=A$ equals $A/M$. For more details we refer to
\cite{BaS} or \cite{MS}.

In \cite{BaS} was also introduced a certain canonical choice of $X$. 
Given $t\in\R$, let $\mathcal P_t=\mathcal P_t(M)$ 
be the convex hull in $\R^n$ of 
$\{(t^{\alpha_1},\ldots, t^{\alpha_n}) \mid z^{\alpha}\in M\}$. 
Then $\mathcal P_t$ is an unbounded 
polyhedron in $\R^n$ of dimension $n$ and the face poset of bounded
faces of $\mathcal P_t$  (i.e., the set of bounded faces
partially ordered by inclusion)  is independent of $t$ if
$t\gg 0$. 
The \emph{hull complex} of $M$ is the polyhedral cell
complex of all bounded faces of $\mathcal P_t$ for $t\gg 0$. 
The corresponding complex $E_\bullet,\varphi_\bullet$ is exact and
thus gives a resolution, the \emph{hull resolution}, of $A/M$. It is
in general not minimal, but it has length at most $n$. 
If $M$ is generic, however, the Hull complex coincides with the Scarf
complex; in particular, it is minimal. 

In \cite{LW} together with L\"ark\"ang we computed the residue current
$R$ associated with the hull resolution, or, more generally, any cellular resolution where the underlying polyhedral complex $X$ is a
polyhedral subdivision of the 
$(n-1)$-simplex, 
of an Artinian monomial ideal. 
Theorem ~5.1 in \cite{LW} states that the entries of $R$ 
are of the form \eqref{poor}, where 
the sum is now over all top-dimensional faces (with label $\alpha$) of
$X$ and $\sgn(\alpha)$ comes from the orientation of $X$.

Note that the definition of $S_{\sigma,\alpha}$ still makes sense when
$M$ is a 
general Artinian monomial ideal. However, in general the 
$S_{\sigma,\alpha}$ will not be cuboids as the following example
shows.

\begin{ex}\label{amsterdam}
Let $M=(z_1^2, z_1z_2, z_1z_3, z_2^2, z_3^2)\subset A_3$. Then $M$ is not
generic, since there is no generator that strictly divides $\lcm
(z_1z_2, z_1z_3)=z_1z_2z_3$. 
The staircase $S$ of $M$ is depicted in Figure ~6.1. 
\begin{figure}\label{hulltrappa}
\begin{center}
\includegraphics{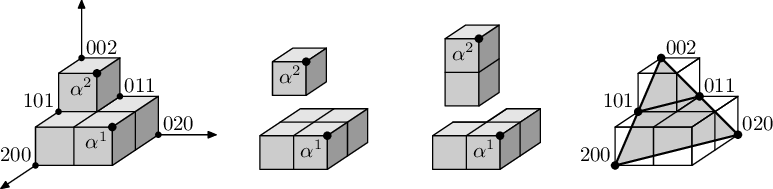}
\caption{The staircase of $M$ 
in Example ~\ref{amsterdam}, the 
  partitions $S_{\alpha^1}$ and $S_{\alpha^2}$ corresponding to the
  orderings $\alpha^1,\alpha^2$ and $\alpha^2,\alpha^1$, respectively,
  and the hull complex of $M$.}
\end{center}
\end{figure}
Note that $S$ has
two outer corners $\alpha^1=(2,2,1)$ and $\alpha^2=(1,1,2)$.

Assume that $\sigma$ is a permutation of $\{1,2,3\}$ such that $\sigma(1)$
equals $1$ or $2$. Then the 
lexicographical order of the outer corners is
$\alpha^1\geq_\sigma\alpha^2$. 
Otherwise, if $\sigma(1)=3$, the lexicographical order is reversed. 
In the first case $S_{\alpha^1}$ is the cuboid
$]0,2]\times]0,2]\times]0,1]$, and $S_{\alpha^2}$ is the cuboid
$]0,1]\times]0,1]\times]1,2]$, see Figure ~6.1. In the
second case $S_{\alpha^2}$ is the cuboid
$]0,1]\times]0,1]\times]0,2]$, whereas $S_{\alpha^1}$ is the set
$]0,2]\times]0,2]\times]0,1]\setminus ]0,1]\times]0,1]\times]0,1]$; in
particular $S_{\alpha^1}$ is not a cuboid.

In this case, the hull resolution is a minimal resolution of $A/M$. There are two top-dimensional faces in
the hull complex, with vertices $\{z_1^2, z_1z_2,
z_1z_3, z_2^2\}$ and $\{z_1z_2, z_1z_3, z_3^2\}$ and thus labels 
$z^{\alpha^1}$ and $z^{\alpha^2}$, respectively, see Figure ~6.1. 
A computation yields that if $\sigma(1)=3$, then the coefficient of
$\sgn (\alpha^1) z^{\alpha^1-\1} dz e^*_{\alpha^1}$ is $3=\Vol (S_{\sigma, \alpha^1})$ and
the coefficient of $\sgn (\alpha^2) z^{\alpha^2-\1} dz e^*_{\alpha^2}$ is
$2=\Vol (S_{\sigma, \alpha^2})$. Otherwise the coefficients are $4=\Vol (S_{\sigma, \alpha^1})$ and
$1=\Vol (S_{\sigma, \alpha^2})$, respectively. 
Thus in this case Theorem ~\ref{main} holds. 
\end{ex}

Example ~\ref{amsterdam} suggests that Theorem ~\ref{main} might hold
when $E_\bullet, \varphi_\bullet$ is the hull resolution of an Artinian 
monomial ideal and this resolution is minimal. However, we
do not know how to prove it in general. 
The proof in Section ~\ref{beviset} does not extend to this
situation. 
For example the staircases $\T_a$ constructed in Section
~\ref{ratblocksbevis} are not disjoint in general, cf.\ \eqref{nollbrutal}. Consider the inner
corners $a=(1,1,0)$ and $b=(1,0,1)$ of the staircase $S$ in Example
~\ref{amsterdam}. 
Then 
\[\T_a\cap \T_b =\{x\in \R^3 |x_1=1, 1<x_j \leq 2, j=2,3\}.\]
Also, the computation of $d\varphi$ is more involved in this
case. Indeed, in general it is not true that the coefficient of
$e_{\I}^*$ in $R$ just consists of one non-vanishing term as in
\eqref{kivas}.

Example ~\ref{motex} below shows
that Theorem ~\ref{main} does not hold for the hull resolution in
general if it is not minimal, 
and also
that it does not hold for arbitrary minimal resolutions of monomial
ideals. 
It would be interesting to look for an alternative description of the
coefficients of $d\varphi$ that extends to general (monomial)
resolutions.

\begin{ex}\label{motex}
Let $M=(z_1^3, z_1^2z_2^2, z_1z_3, z_2^3, z_2z_3, z_3^2)$. Then $M$ is
not generic; for example, as in Example ~\ref{amsterdam}, there is no
generator of $M$ that strictly divides $\lcm (z_1z_2,
z_1z_3)=z_1z_2z_3$. 
The staircase of $M$ has
three outer corners, $\alpha^1=(3,2,1)$, $\alpha^2=(2,3,1)$, 
and $\alpha^3=(1,1,2)$.

In this case the hull resolution is not minimal. The hull complex consists of four triangles; one
triangle $\alpha^j$ for
each outer corner and one extra triangle $\beta$ with vertices $\{z_1^2z_2^2,
z_1z_3, z_2z_3\}$ and thus label $z^\beta=z_1^2z_2^2z_3$, see Figure ~6.2. 
A computation yields that the coefficient of $\sgn(\alpha^3) z^{\alpha^3-\1}dz e^*_{\alpha^3}$ in
$d_\sigma\varphi $ equals $\Vol
(S_{{\sigma, \alpha^3}})$ for all permutations $\sigma$. The coefficient of
$\sgn(\alpha^1) z^{\alpha^1-\1}dz e^*_{\alpha^1}$ in $d_{(3,1,2)}\varphi$,
however, equals $4$, whereas $\Vol(S_{(3,1,2), \alpha^1})=5$. 
Thus Theorem ~\ref{main} does not hold in this case.

\smallskip 

One can create a minimal cellular resolution from the hull
complex, e.g., by removing the edge between $z_1^2z_2^2$ and
$z_1z_3$. The polyhedral cell complex $X$ so obtained has one
top-dimensional face for each outer corner $\alpha^j$ in $S$. The face
corresponding to $\alpha^1$ is the union of the two triangles $\alpha^1$
and $\beta$ in the hull
complex, see Figure ~6.2. 
\begin{figure}\label{hullkomplexena}
\begin{center}
\includegraphics{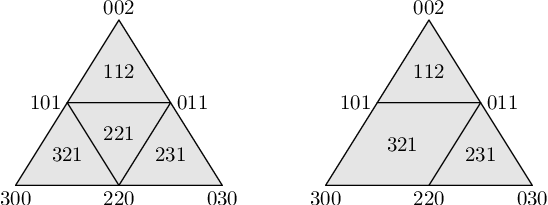}
\caption{The hull complex of the ideal $M$ in Example ~\ref{motex}
  (labels on vertices and $2$-faces) (left) and the minimal free resolution (right).}
\end{center}
\end{figure}
It turns out that the
coefficient of $\sgn(\alpha^1) z^{\alpha^1-\1}dz e^*_{\alpha^1}$ in
$d_\sigma\varphi$ is the sum of the coefficients of $\sgn(\alpha^1) 
z^{\alpha^1-\1}dz e^*_{\alpha^1}$ and
$\sgn(\beta) z^{\beta-\1}dz e^*_\beta$ for each $\sigma$, whereas the coefficients of $\sgn(\alpha^2)z^{\alpha^2-\1}dz e^*_{\alpha^2}$
and $\sgn(\alpha^3) z^{\alpha^3-\1}dz e_{\alpha^3}^*$ are the same as
above. 
Thus, as above, the coefficient of $\sgn(\alpha^2) z^{\alpha^2-\1}dz
e_{\alpha^2}^*$ in $d_{(3,2,1)}\varphi$ is different from 
$\Vol(S_{(3,1,2),\alpha^2})$, and so Theorem ~\ref{main} fails to hold
also in this case.

\end{ex}

Although Theorem ~\ref{main} fails to hold in Example ~\ref{motex},
Corollary ~\ref{foljd} still holds for both resolutions. In fact, the left hand side of
\eqref{fluff} is independent of $\sigma$ for all free resolutions of
Artinian ideals. We will present an argument of this communicated to us by Jan
Stevens, \cite{S}. 

Assume that 
\begin{equation}\label{inteens}
0\to E_n\stackrel{\varphi_n}{\longrightarrow}\ldots\stackrel{\varphi_2}{\longrightarrow} E_1\stackrel{\varphi_1}{\longrightarrow}
E_0\cong \Ok_0
\end{equation}
is a resolution of minimal length of an Artinian ideal
$\a\subset\Ok_0$ and let $R$ be the associated residue current as
constructed in \cite{AW}. 
Since $\a$ is Artinian, it follows from 
the construction that 
\begin{equation}\label{anden}
\varphi_n R=0, 
\end{equation}
see \cite[Proposition~2.2]{AW}. 
Moreover, $R$ satisfies that if 
$\psi$ is (a germ of ) a holomorphic function, then $\psi R=0$ if and only if $\psi \in
\a$, see \cite[Theorem~1.1]{AW}. In particular, 
\begin{equation}\label{borjan}
\varphi_1 \xi R=0
\end{equation} 
for any $\End(E_n,E_2)$-valued section $\xi$.

\begin{prop}\label{artinartin}
Assume that \eqref{inteens} is a resolution of an Artinian ideal
$\a\subset\Ok_0$ and that $z_1,\ldots, z_n$ are 
holomorphic coordinates at $0\in\C^n$. Let $\sigma$ be a permutation of
$\{1,\ldots, n\}$. 
Then \begin{equation*}
\frac{\partial \varphi_1}{\partial z_{\sigma(1)}}dz_{\sigma(1)}\wedge
\cdots \wedge \frac{\partial \varphi_n}{\partial
  z_{\sigma(n)}}dz_{\sigma(n)}\wedge R
\end{equation*} 
is independent of $\sigma$. 
\end{prop}

\begin{proof}
Since each permutation of $\{1,\ldots,n\}$ can be obtained as a
composition of permutations $\sigma$ of the form 
\begin{equation}\label{utan}
\sigma_j:\{1,\ldots, n\}\mapsto \{1,\ldots, j-1,j+1,j,j+2,\ldots, n\},
\end{equation} 
it suffices to prove that 
\begin{equation}\label{pink}
\frac{\partial \varphi_1}{\partial z_{1}}dz_{1}\wedge
\cdots \wedge \frac{\partial \varphi_n}{\partial
  z_{n}}dz_{n}\wedge R =
\frac{\partial \varphi_1}{\partial z_{\sigma_j(1)}}dz_{\sigma_j(1)}\wedge
\cdots \wedge \frac{\partial \varphi_n}{\partial
  z_{\sigma_j(n)}}dz_{\sigma_j(n)}\wedge R.
\end{equation}

Since 
$\varphi_j \varphi_{j+1}=0$, 
\begin{multline}\label{louder}
0=\frac{\partial^2 (\varphi_j \varphi_{j+1})}{\partial z_j\partial
  z_{j+1}}
=
\frac{\partial^2\varphi_j}{\partial z_j\partial
  z_{j+1}} \varphi_{j+1}
+
\frac{\partial\varphi_j}{\partial z_j} \frac{\partial\varphi_{j+1}}
{\partial z_{j+1}}
+
\frac{\partial\varphi_j}{\partial z_{j+1}} \frac{\partial\varphi_{j+1}}
{\partial z_{j}}
+
\varphi_j \frac{\partial^2\varphi_{j+1}}
{\partial z_{j}\partial z_{j+1}}. 
\end{multline} 
Let us compose \eqref{louder} from the left and the right by
\[ 
\frac{\partial\varphi_1}{\partial z_1}\cdots 
\frac{\partial\varphi_{j-1}}{\partial z_{j-1}}
\text{ and } 
\frac{\partial\varphi_{j+2}}{\partial z_{j+2}} \cdots 
\frac{\partial\varphi_{n}}{\partial z_{n}},
\]
respectively. 
Since $\varphi_k \varphi_{k+1}=0$ for each $k$, 
Leibniz's rule gives that 
\begin{equation*}\label{loud}
\frac{\partial\varphi_k}{\partial z_\ell} \varphi_{k+1} = -
\varphi_k\frac{\varphi_{k+1}}{\partial z_\ell},
\end{equation*}
cf.\ \eqref{louder}. 
Using this repeatedly for $k=j+1,\ldots, n-1$ we get that the 
term corresponding to the first term in the left 
hand side of \eqref{louder} equals 
\[
\pm\frac{\partial\varphi_1}{\partial z_1}\cdots 
\frac{\partial\varphi_{j-1}}{\partial z_{j-1}}  
\frac{\partial^2\varphi_j}{\partial z_j\partial
  z_{j+1}} 
\frac{\partial\varphi_{j+1}}{\partial z_{j+2}} \cdots 
\frac{\partial\varphi_{n-1}}{\partial z_{n}}\varphi_n.
\]
Similarly the term corresponding to the last term in the left hand
side of \eqref{louder} equals  
\[
\pm\varphi_1 
\frac{\partial\varphi_2}{\partial z_1} \cdots 
\frac{\partial\varphi_{j}}{\partial z_{j-1}}  
\frac{\partial^2\varphi_{j+1}}{\partial z_j\partial
  z_{j+1}} 
\frac{\partial\varphi_{j+2}}{\partial z_{j+2}}\cdots 
\frac{\partial\varphi_{n}}{\partial z_{n}}.
\]
Next let us compose from the right by $R$. 
Using \eqref{anden} and \eqref{borjan} we get 
\begin{equation*} 
0=\frac{\partial \varphi_1}{\partial z_1} 
\cdots \frac{\partial \varphi_n}{\partial
  z_n} R + 
\frac{\partial \varphi_1}{\partial z_1}\cdots
\frac{\partial\varphi_{j}}{\partial z_{j+1}}
\frac{\partial\varphi_{j+1}}{\partial z_j}\cdots
\frac{\partial \varphi_n}{\partial z_n} R =
\frac{\partial \varphi_1}{\partial z_1} 
\cdots \frac{\partial \varphi_n}{\partial
  z_n} R + 
\frac{\partial \varphi_1}{\partial z_{\sigma_j(1)}}\cdots
\frac{\partial \varphi_n}{\partial z_{\sigma_j(n)}} R.
\end{equation*} 
Combining this with $dz_1\wedge\cdots\wedge
dz_n=-dz_{\sigma_j(1)}\wedge\cdots \wedge dz_{\sigma_j(n)}$ 
we obtain \eqref{pink}. 
\end{proof}

\def\listing#1#2#3{{\sc #1}:\ {\it #2},\ #3.}

\end{document}